\documentclass[a4wider]{amsart}
\usepackage{enumerate}
\usepackage{fontenc}
\usepackage{color}
\usepackage{amsfonts,amsmath,amssymb,amsthm}
\usepackage{dsfont}
\usepackage{latexsym}
\usepackage{enumerate}
\usepackage[T1]{fontenc}
\usepackage{yfonts}
\usepackage{amsmath,amsfonts,amstext}
\usepackage{amssymb,amsmath,amscd,graphicx,epsfig,hhline,mathrsfs,yhmath,latexsym,url}
\usepackage[arrow, matrix, curve]{xy}
\usepackage[british]{babel}
\usepackage{paralist}               	
\newcommand{\mk}{\mathfrak}
\newcommand{\mc}{\mathcal}
\newcommand{\ms}{\mathscr}

\newcommand{\mf}{\mathbf}
\newcommand{\mb}{\mathbb}
\newcommand{\mr}{\mathrm}

\newtheorem{theorem}{Theorem}[section]
\newtheorem{corollary}[theorem]{Corollary}

\newtheorem{definition}[theorem]{Definition}

\newtheorem{lemma}[theorem]{Lemma}
\newtheorem{proposition}[theorem]{Proposition}

\newtheorem{remarka}[theorem]{Remark}
\newenvironment{remark}{\begin{remarka}\rm}{\hfill\rule{2mm}{2mm}\end{remarka}}

\newtheorem*{T*}{Theorem}
\newtheorem*{A*}{Proposition}
\newtheorem*{Cor*}{Corollary}
\newtheorem{T}[theorem]{Theorem}
\newtheorem{Le}[theorem]{Lemma}
\newtheorem{R}[theorem]{Remark}
\newtheorem{Cor}[theorem]{Corollary}
\newtheorem{A}[theorem]{Proposition}
\theoremstyle{definitionbreak}\newtheorem{D}[theorem]{Definition}
\theoremstyle{definitionbreak}\newtheorem*{D*}{Definition}
\newtheorem{Ex}[theorem]{Example}

\DeclareMathOperator{\Exists}{\exists}
\DeclareMathOperator{\Forall}{\forall}
\newtheorem{prelem}{{\bf Theorem}}

\def\dist {\mathrm{dist}}

\title{Measures on the square as sparse graph limits}
\author{D\'avid Kunszenti-Kov\'acs, L\'aszl\'o Lov\'asz, Bal\'azs Szegedy}

\begin{document}		
\maketitle

\begin{abstract} We study a metric on the set of finite graphs in which two graphs are considered to be similar if they have similar bounded dimensional ``factors''. We show that limits of convergent graph sequences in this metric can be represented by symmetric Borel measures on $[0,1]^2$. This leads to a generalization of dense graph limit theory to sparse graph sequences. 
\end{abstract}

\section{Introduction}

It was proved in \cite{LSz} that if all subgraph densities converge in a growing graph sequence then there is a natural limit object in the form of a symmetric measurable function $W:[0,1]^2\rightarrow [0,1]$ called {\it graphon}. If the edge density converges to $0$ then all other subgraph densities converge to $0$ and the limit object is the constant $0$ function. Such graph sequences are called {\it sparse}. Many naturally occurring graph sequences are sparse and capturing their limiting structure remains one of the great challenges in graph limit theory. 

In the very sparse case, when the maximum degree is uniformly bounded in the sequence, one can use the so-called Benjamini-Schramm convergence \cite{BS} or a refinement of it called {\it local-global convergence} \cite{BR},\cite{HLSz}. The limit object can be represented by a bounded degree Borel graph (called {\it graphing}) satisfying a certain measure preserving property.  

Both in the dense case and in the bounded degree case the limit object can be equivalently viewed as a symmetric Borel measure $\mu$ on $[0,1]^2$. (Symmetry of $\mu$ means that it is invariant under the map $\tau:[0,1]^2\rightarrow [0,1]^2$ given by $\tau(x,y)=(y,x)$.) In the dense case the graphon $W$ is the Radon-Nikodym derivative of some Borel measure $\mu$ on $[0,1]^2$. In the bounded degree case the edge set of the graphing (when represented on the vertex space $[0,1]$) is essentially equal to the support of some Borel probability measure on $[0,1]^2$ which is uniquely determined by the property that it is uniform on the edge set. The goal of this paper is to develop a limit theory for graphs in which an arbitrary graph sequence has an appropriately convergent subsequence with a limit object of the form of a symmetric Borel measure on $[0,1]^2$.

Our main tool is a variant of Szemer\'edi's regularity lemma \cite{Szem} for sparse graphs that allows us to approximate the set of all possible bounded (say $k$) dimensional, normalized factor matrices (called $k$-{\it shape}) of a given graph. (Other sparse versions of the regularity lemma were proved in \cite{K},\cite{Sc}.)
The study of $k$-shapes goes back to dense graph limit theory where it was observed that the $k$-shape of a graph (when the factor matrices are normalized with the square of the number of vertices) determines the density matrix of a Szemer\'edi partition with precision depending on $k$ and thus it determines the densities of small subgraphs (an even stronger theorem was proved in \cite{BCLSV}). This implies that one can equivalently define dense graph convergence through the convergence of all $k$-shapes in the Hausdorff metric. However if the factor matrices are normalized with the number of edges (multiplied by $2$ for technical reasons) then the convergence of $k$-shapes leads to a non-trivial limit notion (we call it {\it s-convergence}) for sparse graph sequences. Informally speaking our main result is the following.

{\it If all the (edge number normalized) $k$-shapes converge in a graph sequence then the limit can be represented by some symmetric Borel probability measure on $[0,1]^2$ (called an s-graphon). Furthermore every s-graphon arises this way.}

Our main theorem (see Theorem~\ref{thm:shapeMeasure}) is actually slightly more general. It deals with the limits of weighted graphs (non-negative symmetric matrices). It is important to mention that most of our statements and arguments use the cantor set $\mathcal{C}=\{0,1\}^{\mathbb{N}}$ instead of the unit interval $[0,1]$ and the equivalence between the two is proved in Chapter \ref{ch:cantor}. The advantage of using the Cantor set in the proofs is due to its more combinatorial nature. 

Our limit theory can also be considered as a generalization of the so-called $L^p$ theory of sparse graph convergence (see \cite{BCCZ} and \cite{BCCZ2}). Roughly speaking, If a graph sequence converges in the $L^p$ theory (for some fixed $p$) then the sequence also converges in our language and the limit object appears as the Radon-Nikodym derivative of the limit object measure on $[0,1]^2$. The only feature that we lose in general is the connection to subgraph densities. This is the price that we pay for the generality of the limit concept. In exchange we get a representation of all graphs in a compact space which detects non-trivial structure in arbitrarily sparse graphs. This tradeoff is justified in Chapter \ref{ch:ash}, where we investigate in detail how this gained compactness impacts the limit results pertaining to limit objects that are $L^p$ graphons. We relate our approach to the theory developed by Borgs et al., highlighting that both bring important added value to the understanding of graph limits in the non-dense case. 

In our theory, there are examples of convergent graph sequences such that the limit measure on $[0,1]^2$ is not absolutely continuous with respect to the uniform measure and in these cases we can't associate a measurable function with the limit object. We can get limit objects that are measures concentrated on fractal like subsets in $[0,1]^2$.  We propose a possible dimension notion accompanying our limit theory that associates fractional dimensions between $0$ and $2$ with measures on $[0,1]^2$. 

The paper is structured as follows. The definition of graph convergence is introduced in Chapter \ref{ch:conv}. In the same chapter we state two versions of the main theorem: Theorem \ref{thm:shapeMeasure} and Theorem \ref{thm:shpm2}. Chapter \ref{ch:reg} contains the statement and the proof of our general regularity lemma.
The main theorem is proved in Chapter \ref{ch:main}. Chapter \ref{ch:ash} and Chapter \ref{ch:closed} deals mostly with the connections to $L^p$ limit theory. Chapter \ref{ch:conc} investigates various aspects of our limit theory and Chapter \ref{ch:ex} has examples for interesting sparse sequences.

\section{Basics}

\noindent{\bf On the Cantor set:}~~Let $\mathcal{C}=\{0,1\}^{\mathbb{N}}$ be the power of the discrete
topological space $\{0,1\}$. The space $\mathcal{C}$ carries a
natural probability measure $\nu$, namely the product measure of the uniform
measure on $\{0,1\}$. The elements of $\mathcal{C}$ can be represented by infinite pathes in an infinite rooted binary tree $B$. A node
$x\in V(B)$ of $B$ represents an open-closed subset $S_x$ of $\mathcal{C}$
by collecting all the pathes going through $x$. These sets form a basis for the topology on $\mathcal{C}$. By
Tychonoff's theorem, $\mathcal{C}$ is a compact space, and so the
subsets $S_x$ are compact too. In fact, even stronger properties hold.

\begin{proposition}\label{Radon}
Both $\mathcal{C}$ and $\mathcal{C}\times \mathcal{C}$ are compact Polish spaces, and equipped with their Borel $\sigma$-algebra, they are Radon spaces (i.e., any finite Borel measure is automatically inner regular, and so a Radon measure).
\end{proposition}

\begin{lemma}\label{mertek}
Let $\mu:V(B)\rightarrow \mathbb{R}^+\cup\{0\}$ satisfy $\mu(x)=\mu(x_1)+\mu(x_2)$
for all $x\in V(B)$ where $x_1$ and $x_2$ are the two
children of the node $x$. Then the map $S_x\mapsto \mu(x)$ extends to a Borel measure on
$\mathcal{C}$.
\end{lemma}
\begin{proof}
Let $S$ be the ring generated by the set system $\{S_x \ | \ x\in
V(B)\}$. It is easy to see that every element of $S$ can be expressed as a finite
disjoint union $\cup_{i\in I}S_{x_i}$ where the nodes $x_i$ are at
the same level of the tree $B$. Using the additivity condition,
$\mu$ extends to $S$ as a finitely-additive function. We claim that
$\mu$ is $\sigma$-additive on $S$. Let $s\in S$ be a
disjoint union of countably many elements from $S$. Since $s$ is a
compact set and each element of $S$ is open, this union
has finitely many non-empty terms. Now we can use Caratheodory's
extension theorem to obtain that $\mu$ extends to the
$\sigma$-algebra generated by $S$ as a measure. Since $S$ is a basis
for the topology, we get that $\mu$ extends to a measure on the
Borel sets.
\end{proof}

\medskip

\noindent{\bf Limits of tables:}~~Let $\{T_n\}_{n=0}^\infty$ be an infinite sequence of nonnegative
real matrices such that $T_n$ is of size $2^n\times 2^n$. The rows and
columns of $T_n$ are indexed by the nodes from the $n$-th level of
the rooted binary tree $B$. We denote by $T_n(x,y)$ the entry of
$T_n$ in the intersection of the row $x$ and column $y$.

\begin{definition}
The system $\{T_n\}_{n=0}^\infty$ is called {\bf consistent} if
$$T_n(x,y)=T_{n+1}(x_1,y_1)+T_{n+1}(x_1,y_2)+T_{n+1}(x_2,y_1)+T_{n+1}(x_2,y_2),$$
for all $n=0,1,2,\ldots$, and $x,y$ where $x_1,x_2$, and $y_1,y_2$ are respectfully the
children of $x$ and $y$.
\end{definition}

\begin{lemma}\label{tabcon}
If $\{T_n\}_{n=0}^\infty$ is consistent,  then there is a measure
$\mu$ on $\mathcal{C}^2$ such that $$\mu(S_x\times S_y)=T_n(x,y)$$
where $x$ and $y$ are at the $n$-th level of the binary tree $B$.
\end{lemma}

\begin{proof}
Let us introduce the product tree $B\times B$ whose nodes are
ordered pairs $(x,y)$ where $x$ and $y$ are at the same level in
$B$. The node $(x,y)$ has $4$ children:
$$(x_1,y_1),(x_2,y_1),(x_1,y_2),(x_2,y_2)$$ where $x_1,x_2$, and $y_1,y_2$
are the children of $x$ and $y$, respectfully. A consistent sequence $\{T_n\}_{n=0}^\infty$
defines a function $\mu: B\times B\rightarrow\mathbb{R}^+\cup\{0\}$
by $\mu: (x,y) \mapsto T_n(x,y)$. This function is analogous to the one in Lemma~\ref{mertek}, and
using the same proof as in Lemma~\ref{mertek}, the function
$$\mu:(S_x\times S_y)\rightarrow\mu(x,y)$$ extends to a Borel
measure on $\mathcal{C}\times\mathcal{C}$.
\end{proof}

\section{Balanced partitions}\label{ch:bp}

Let $\Omega$ be a finite set of $n$ elements. We say that
$\Omega=\cup_{i=1}^k\Omega_i$ is a balanced partition of $\Omega$, if
$|\Omega_i|\in\{\lceil n/k \rceil , \lfloor n/k \rfloor\}$ for all
$1\leq i \leq k$. It is easy to see that every finite set has a
balanced partition for every natural number $k$, and the multi set
$\{|\Omega_1|,|\Omega_2|,\ldots,|\Omega_k|\}$ is uniquely determined
by $n$ and $k$. Each partition of $\Omega$ has a characteristic $k \times |\Omega|$
matrix $M$ whose rows are the characteristic vectors of the
partition sets. Each column sum of $M$ is $1$, and if $n$ divides $k$ and the partition is balanced,
then each row sum is $n/k$. We
denote the set of all balanced $k \times n$ partition matrices by
$\widehat{\mathcal{K}}(k,n)$.

Let $S$ be a non-negative matrix whose rows and columns are indexed
by a finite set $\Omega$ with $n$ elements and let $k$ be a fixed
natural number. For a balanced partition
$\mathcal{P}=\{\Omega_1,\Omega_2,\dots,\Omega_k\}$ of $\Omega$, we
define a matrix $\mathcal{P}(S)$ whose $ij$-th entry is $$\sum_{x\in\Omega_i, y\in\Omega_j}S_{x,y}.$$ We
denote by $\widehat{C}(S,k)$ the set of all matrices that can be
obtained as $\mathcal{P}(S)$ for some balanced partition
$\mathcal{P}$ with $k$ sets.

Let $\mathcal{K}(k,n)$ denote the set of all nonnegative $k$-by-$n$
matrices with each row sum equal to $n/k$ and  each column sum equal to
$1$. For an arbitrary $n\times n$ matrix $S$ and a natural number
$k$ we define a shape $C(S,k)$ in $\mathbb{R}^{k\times k}$ in the
following way:

$$C(S,k):=\{MSM^T~|~M\in \mathcal{K}(k,n)\}.$$ Since $\mathcal{K}(k,n)$ is compact,
we have that $C(S,k)$ is a compact subset of the set of real $k$ by $k$
matrices. It is also important to mention that the set $C(S,k)$ is
invariant under conjugation with permutation matrices. For an
arbitrary set of matrices $\mathscr{S}$, we define $C(\mathscr{S},k)$
as
\[
C(\mathscr{S},k):=\bigcup_{s\in\mathscr{S}}C(S,k).
\]
If $\mathscr{S}$ is a compact subset of $n\times n$ matrices, then
$C(\mathscr{S},k)$ is also a compact set which depends continuously
on $\mathscr{S}$.

We can define an analogy of $C(S,k)$ for a Borel measure $\mu$ on
$\mathcal{C}\times\mathcal{C}$ with finite total measure. Let
$f_1,f_2,\dots,f_k$ be nonnegative (Borel) measurable functions
on $\mathcal{C}$ with the property that their sum is the
constant function $1$ and that for all $1\leq i\leq k$,
\[
\int_{\mathcal{C}} f_i~ d\nu=1/k.
\]
For such a sequence of functions, we define a $k\times k$ matrix $M$ by setting
\[
M(i,j):=\int_{(x,y)\in \mathcal{C}\times\mathcal{C}}f_i(x)f_j(y) d\mu.
\]
Sometimes we denote this matrix by $\mathcal{M}(f_1,f_2,\dots,f_k)$. Let so
$C_0(\mu,k)$ be the set of all matrices $\mathcal{M}(f_1,f_2,\dots,f_k)$ for all possible choices of suitable
functions $f_1,f_2,\dots,f_k$. We denote by $C(\mu,k)$ the topological closure of
$C_0(\mu,k)$.

In addition to these measurable shapes, it will sometimes be more convenient to work with continuous functions only. Let therefore $C_c(\mu,k)$ denote the set of all matrices $\mathcal{M}(f_1,f_2,\dots,f_k)$ where the functions $f_i$ are in addition continuous on $\mathcal{C}$.

\begin{lemma}
For any symmetric finite Borel measure $\mu$ on $\mathcal{C}\times \mathcal{C}$ we have that $C_0(\mu,k)\subseteq \overline{C_c(\mu,k)}$ for every natural number $k$.
If $\mu$ is absolutely continuous with respect to $\nu\times\nu$, then for every natural number $k$, the set $C_0(\mu,k)$ is closed, and so $C_0(\mu,k)=C(\mu,k)$.
\end{lemma}
\begin{proof}
For the first part, let $X\in C_0(\mu,k)$, and suppose that $f_1,f_2,\ldots,f_k$ are corresponding measurable witnesses. Further, let $\mu^1$ denote the marginal of $\mu$ on $\mathcal{C}$ (since $\mu$ is symmetric, its two marginals are equal). Now, since by Proposition \ref{Radon} $(\mathcal{C},\mathscr{B}(\mathcal{C}),\mu^1)$ is a Radon measure space on a compact set, by Lusin's theorem we can for any $\varepsilon>0$ find a compact set $E\subset\mathcal{C}$ with $\mu^1(\mathcal{C}\backslash E)<\varepsilon$ and $\nu(\mathcal{C}\backslash E)<\varepsilon$, together with continuous functions $g_1,g_2,\ldots,g_k$ on $\mathcal{C}$ taking values in $[0,1]$ such that for each $1\leq i\leq k$ we have
$g_i\vert_E=f_i\vert_E$. It is not hard to see that since $\sum_{i=1}^k f_i\equiv 1$ is a continuous function, we may assume that we also have $\sum_{i=1}^k g_i\equiv 1$. Since each $g_i$ differs from the corresponding $f_i$ on a set of small measure only, we have that
\[
\left|\int g_i d\nu-\frac{1}{k}\right|=\left|\int (f_i-g_i) d\nu\right|\leq \varepsilon
\]
for each $i$. We would like to turn the family $g_1,g_2,\ldots,g_k$ into a witness for an element in $C_c(\mu,k)$, but their integrals still don't quite match up. We shall therefore iteratively change the $g_i$-s in such a way that they remain continuous and nonnegative, and also their sum is still the constant 1 function, but after each step at least one additional function will have integral exactly $1/k$. Thus in at most $k-1$ steps we shall end up with a witness family.\\
Given a continuous function $F:\mathcal{C}\to[0,1]$ and a parameter $\alpha\in[0,1]$, define the continuous function
\[
F^\alpha(x):=\min\{\alpha,F(x)\}.
\]
This has the property that the integral $\int_\mathcal{C} F^\alpha d\nu$ is a continuous, monotone increasing function of $\alpha$. Suppose now that for some $1\leq\ell\leq k$ the function $g_\ell$ has an integral greater than $1/k$ (unless all are exactly $1/k$, at least one such $\ell$ exists). 

Since 
\[
\nu([g_\ell>\alpha])+\alpha(1-\nu([g_\ell>\alpha]))\geq \int_\mathcal{C} g_\ell^{\alpha} d\nu \geq \frac{1}{k},
\]
we have 
\[
\nu([g_\ell>\alpha])\geq \frac{\frac{1}{k}-\alpha}{1-\alpha},
\]
and so
\[
\int_\mathcal{C} g_\ell^{2k\varepsilon} d\nu\geq 2k\varepsilon \frac{\frac{1}{k}-2k\varepsilon}{1-2k\varepsilon}
=\varepsilon \frac{2-4k^2\varepsilon}{1-2k\varepsilon}>\varepsilon
\]
for whenever $\varepsilon<\frac{1}{4k^2}$. Then, by continuity of the integral as a function of $\alpha$, there exists a $0<\beta<2k\varepsilon$ such that 
\[
\int_{\mathcal{C}} g_\ell^\beta d\nu=\int_{\mathcal{C}} g_\ell d\nu -\frac{1}{k},
\]
whence $g_\ell':=g_\ell-g_\ell^\beta$ is a nonnegative function with integral exactly $1/k$. Replacing $g_\ell$ by $g_\ell'$, and one of the other $g_i$-s with integral not equal to $1/k$ -- say $g_j$ -- by $g_j+g_\ell^\beta$, we obtain a new family of functions with the desired properties.
Note that during one such step, no function changed by more than $2k\varepsilon$ in $\|\cdot\|_\infty$-norm.\\
Thus at the end of the process -- after at most $(k-1)$ steps -- we obtain a family of nonnegative continuous functions $(\widetilde{g}_1,\widetilde{g}_2,\ldots,\widetilde{g}_k)$ that generate the element
$Y:=\mathcal{M}(\widetilde{g}_1,\widetilde{g}_2,\ldots,\widetilde{g}_k)\in C_c(\mu,k)$, and such that for each $1\leq i\leq k$ we have 
\[
\|g_i-\widetilde{g}_i\|_\infty<(k-1)\cdot2k\varepsilon<2k^2\varepsilon.
\]
This means that for any pair of indices $j_1,j_2$, we have with the notation $D_{j_1,j_2}(x,y):=f_{j_1}(x)f_{j_2}(y)-\widetilde{g}_{j_1}(x)\widetilde{g}_{j_2}(y)$ that
\begin{align*}
\left|X_{j_1,j_2}-Y_{j_1,j_2}\right|&=
\left| \int_{(x,y)\in\mathcal{C}^2}
D_{j_1,j_2}(x,y)~d\mu
\right|\leq
\left| \int_{(x,y)\in E^2}
D_{j_1,j_2}(x,y)~d\mu
\right|\\&+
\int_{(x,y)\in \left(\mathcal{C}\backslash E\right)\times E}
\left| D_{j_1,j_2}(x,y)\right|~d\mu
+
\int_{(x,y)\in \left(E\times \mathcal{C}\backslash E\right)}
\left| D_{j_1,j_2}(x,y)\right|~d\mu\\
&\leq
\mu(E^2)\cdot \|f_{j_1}-\widetilde{g}_{j_1}\|_\infty\cdot \|f_{j_2}-\widetilde{g}_{j_2}\|_\infty
+2\mu(E\times\left(\mathcal{C}\backslash E\right))\\
&\leq 4k^4\varepsilon^2 +2 \mu^1(E)\left(1-\mu^1(E)\right)\leq \varepsilon(4k^4+2),
\end{align*}
hence $\|X-Y\|_\infty\leq \varepsilon(4k^4+2)$.
Since $\varepsilon$ can be chosen to be arbitrarily small, we obtain that indeed $X\in\overline{C_c(\mu,k)}$.

\noindent Let us now turn to the second part. Suppose that the sequence $(X_i)\subset C_0(\mu,k)$ is convergent.
Then for every natural number $i$ we have a family of nonnegative witness functions
$f_{1,i},f_{2,i},\dots,f_{k,i}$ with $\sum_{j=1}^{k}f_{j,i}=1$~,~$\int_{\mathcal{C}}f_{j,i}~d\nu =1/k$
such that for each pair $j_1,j_2$, as $i$ goes to infinity, the integrals $$\int_{(x,y)\in
\mathcal{C}^2}f_{j_1,i}(x)f_{j_2,i}(y)~d\mu$$ converge to
some fixed value.

 We can choose an infinite sequence
$r_1,r_2,\dots$ from the natural numbers such that for all $j,\ell$ the
sequences $f_{j,r_i}$ and $f_{j,r_i}(x)f_{\ell,r_i}(y)$ are weak-* convergent in
$L^\infty(\nu)$ and $L^\infty(\nu\times\nu)$, respectively. Now let $g_{j}$ be the weak-* limit of $f_{j,r_i}$, and $g_{j,\ell}$ that of $f_{j,r_i}(x)f_{\ell,r_i}(y)$.
Note that clearly $\int_{\mathcal{C}}g_{j}~d\nu =1/k$, but also $\int_{A}\sum_{j=1}^k g_{j}~d\nu =\nu(A)$ for all $A\in\mathscr{B}(\mathcal{C})$, and so $\sum_{j=1}^{k}g_{j}=1$.
Thus this family induces a matrix $X:=\mathcal{M}(g_1,g_2,\ldots g_k)\in C_0(\mu,k)$. Our aim is to show that $X$ is the limit of the matrices $X_i$.
\\
Consider the Radon-Nikodym derivative $W\in L^1(\nu\times\nu)$ of the measure $\mu$, which is nonnegative and symmetric by the assumption on $\mu$. 
First, we show that $g_{j,\ell}(x,y)=g_j(x)g_\ell(y)$ holds $\nu\times\nu$-almost everywhere. Since the characteristic functions $\chi_{A\times B}$ with $A,B\in\mathscr{B}(\mathcal{C})$ generate a dense subspace in $L^1(\nu\times\nu)$, this amounts to the weak evaluations
\[
\langle \chi_{A\times B}, g_{j,\ell}\rangle_{L^1(\nu\times\nu),L^\infty(\nu\times\nu)}
\]
and
\[
\langle \chi_{A\times B}, g_j\otimes g_\ell\rangle_{L^1(\nu\times\nu),L^\infty(\nu\times\nu)}
\]
being equal. But this follows from
\[
\langle \chi_{A\times B}, g_j\otimes g_\ell\rangle_{L^1(\nu\times\nu),L^\infty(\nu\times\nu)}=
\langle \chi_{A}, g_j\rangle_{L^1(\nu),L^\infty(\nu)}
\cdot
\langle \chi_{B}, g_\ell\rangle_{L^1(\nu\times\nu),L^\infty(\nu\times\nu)}
\]
and taking the limit in the equalities
\begin{eqnarray*}
&&\langle \chi_{A\times B}, f_{j,r_i}\otimes f_{\ell,r_i}\rangle_{L^1(\nu\times\nu),L^\infty(\nu\times\nu)}
=\int_{A\times B} f_{j,r_i}\otimes f_{\ell,r_i}d\nu\times\nu\\
&=&
\left(\int_{A} f_{j,r_i} d\nu\right)\cdot\left(\int_{B} f_{\ell,r_i} d\nu\right)
=\langle \chi_{A}, f_{j,r_i}\rangle_{L^1(\nu),L^\infty(\nu)}\cdot \langle \chi_{B}, f_{\ell,r_i}\rangle_{L^1(\nu),L^\infty(\nu)}
\end{eqnarray*}

But then we have
\begin{align*}
&X(j,\ell)=\int_{\mathcal{C}\times \mathcal{C}} g_{j}\otimes g_{\ell} d\mu=
\int_{\mathcal{C}\times \mathcal{C}} g_{j,\ell} d\mu=\int_{\mathcal{C}\times \mathcal{C}} g_{j,\ell}\cdot W d\nu\times\nu\\
&=
\lim_{i\to\infty}\int_{\mathcal{C}\times \mathcal{C}} f_{j,r_i}\otimes f_{\ell,r_i}\cdot W d\nu\times\nu=
\lim_{i\to\infty}\int_{\mathcal{C}\times \mathcal{C}} f_{j,r_i}\otimes f_{\ell,r_i} d\mu=\lim_{i\to\infty} X_{r_i} (j,\ell),
\end{align*}
concluding the proof.
\end{proof}

\section{The notion of convergence and the representation of the limit}\label{ch:conv}
Since the sum of the entries of matrices $S$ will play an important role, we introduce the short hand notation $\gamma(S)$ for this quantity. It is easy to see that $\gamma(S)=\gamma(X)$ for all $k$ and $X\in C(S,k)$.\\

\begin{definition}\label{def:matconv}
Let $c>0$ be a constant and $\{S_i\}_{i=1}^\infty$ be a sequence of non-negative matrices that satisfy $\gamma(S_i)<c$ for every $i$. We say that the sequence $\{S_i\}_{i=1}^\infty$ is convergent if  for every natural number $k$, the shapes
$C(S_i,k)$ are converging to some fixed closed set in
$[0,c]^{k\times k}$ with respect to the Hausdorff topology.
\end{definition}

Note that when speaking of the Hausdorff distance of compact subsets of $\mathbb{R}^{k\times k}$, we mean the metric $\mathrm{dist}(\cdot,\cdot)$ induced by the $\ell_1$ norm of $\mathbb{R}^{k\times k}$.
Now we can formulate our main theorem regarding the global limits.

\begin{theorem}\label{thm:shapeMeasure}
If $\{S_i\}_{i=1}^\infty$ is a convergent sequence of non-negative symmetric matrices then there
exists a symmetric Borel measure $\mu$ on $\mathcal{C}\times\mathcal{C}$
such that for all natural numbers $k$, the limit shape of $C(S_i,k)$ is $C(\mu,k)$.
\end{theorem}

As a corollary (For a detailed proof see Chapter \ref{ch:cantor}) we get the following statement.

\begin{corollary}\label{mainCor}
If $\{S_i\}_{i=1}^\infty$ is a convergent sequence of non-negative symmetric matrices, then there
exists a symmetric Borel measure $\mu$ on $[0,1]^2$
such that for all natural number $k$, the limit shape of $C(S_i,k)$ is $C(\mu,k)$.
\end{corollary}

For a finite graph $G$ (with non empty edge set) we define $C_0(G,k)$ to be
$\|A_G\|_1^{-1}C(A_G,k)$ where $A_G$ is the adjacency matrix of
$G$. Note that if $G$ has no loop edges than $\|A_G\|_1=2|E(G)|$.

\begin{definition}
We say that the sequence $\{G_i\}_{i=1}^\infty$ of graphs is s-convergent if for every fixed natural number $k$, the shapes
$C_0(G_i,k)$ are converging in the Hausdorff metric. 
\end{definition}

It is clear that the above convergence notion puts all non-empty graphs into a compact space since every sequence has an s-convergent sub-sequence. The following theorem is an immediate consequence of Theorem~\ref{thm:shapeMeasure}.

\begin{theorem}\label{thm:shpm2}
If $\{G_i\}_{i=1}^\infty$ is an s-convergent sequence of graphs, then there is a
Borel probability measure $\mu$ on $[0,1]^2$ such that the limit shape of
$C_0(S,k)$ is $C(\mu,k)$ for all natural number $k$. 
\end{theorem}

Motivated by Theorem \ref{thm:shpm2} we introduce the following notion.

\begin{definition} An s-graphon is a  symmetric Borel probability measure on $[0,1]^2$.
\end{definition}

We will prove that every s-graphon is a limit of some graph sequence.

\begin{theorem}\label{thm:shpm2b} Let $\mu$ be an s-graphon. then there is a graph sequence $\{G_i\}_{i=1}^\infty$ such that it is s-convergent and its limit is $\mu$.
\end{theorem}

\begin{corollary}\label{cor:shpm2c} Let $\mu$ be an arbitrary symmetric Borel measure on $[0,1]^2$ (or on $\mathcal{C}^2$). Then there is a sequence of non-negative symmetric matrices $\{S_i\}_{i=1}^\infty$ with $\gamma(S_i)=\mu([0,1]^2)$ such that the limit of $\{S_i\}_{i=1}^\infty$ is $\mu$.
\end{corollary}

\begin{proof} Let $\mu'=\mu/\mu([0,1]^2)$. We have by theorem \ref{thm:shpm2b} that $\mu'$ is the limit of the matrices $A_{G_i}/\|A_{G_i}\|_1$ for some graph sequence $\{G_i\}_{i=1}^\infty$. It follows that $\mu$ is the limit of the matrices $A_{G_i}\mu([0,1]^2)/\|A_{G_i}\|_1$.
\end{proof}

\section{Some lemmas}

\begin{lemma}\label{folytrep}
Let $S$ be a nonnegative $n \times n$ matrix. There is a measure
$\mu$ on $\mathcal{C}^2$ such that $C(\mu,k)=C(S,k)$ for every
natural number $k$.
\end{lemma}
\begin{proof}
First we represent $S$ by a measurable function $g$ on
$\mathcal{C}^2$ in the following way. We partition $\mathcal{C}$
into $n$ subsets $\mathcal{P}_1,\mathcal{P}_2,\dots,\mathcal{P}_n$
each of measure $1/n$. For a point $x\in
\mathcal{P}_i\times\mathcal{P}_j$, we define $g(x)$ to be
$n^2S_{i,j}$. Now the function $g$ defines a measure $\mu$ on
$\mathcal{C}^2$ by
$$\mu(H)=\int_{H}g~d(\nu\times\nu).$$ It is easy to see that
$C(\mu,k)=C(S,k)$ for every natural number $k$.
\end{proof}

\begin{lemma}\label{resze}
Let $\mu$ be a measure on $\mathcal{C}^2$ and let $X$ be an element
of $C_0(\mu,k)$ for some natural number $k$. Then $C(X,r)\subseteq
C(\mu,r)$ for every natural number $r$.
\end{lemma}
\begin{proof}
Let $Y$ be an arbitrary element of $C(X,r)$. We know by definition
that $Y=MXM^T$ for some nonnegative real $r\times k$ matrix
$M\in\mathcal{K}(k,n)$. Let $X_1,X_2,\dots$ be a sequence of matrices in
$C_0(\mu,k)$ converging to $X$, and for each natural number $n$, let
$f_{n,1},f_{n,2},\dots,f_{n,k}$ be a systems of measurable functions
on $\mathcal{C}$ that is a witness for $X_n\in C_0(\mu,k)$. That is $X_n=\mathcal{M}(f_{n,1},f_{n,2},\dots,f_{n,k})$,  the sum $\sum_{i=1}^k f_{n,i}$ is
the constant function $1$ and  for all $1\leq i\leq k$,
$$\int_{\mathcal{C}} f_{n,i}~ d\nu=1/k.$$
Define
$$g_{n,i}=\sum_{j=1}^{k}M_{i,j}f_{n,j}$$ and let $G_n$ be an $r \times r$ matrix
with entries
$$G_n(i,j)=\int_{(x,y)\in\mathcal{C}^2}g_{n,i}(x)g_{n,j}(y)~d(\nu\times\nu).$$
Since $G_n=MX_nM^T$, we get that the sequence $G_n$ is
converging to $Y$. On the other hand, the definition of $G_n$ shows
that $G_n\in C_0(\mu,r)$. This proves that $Y\in C(\mu,r)$.
\end{proof}

This result has the following consequence for shapes pertaining to the same measure/matrix.

\begin{corollary}\label{tart}
Let $\mu$ be a measure on $\mathcal{C}^2$ of finite total measure.
Then for all natural numbers $k$ and $r$ we have that $C(C(\mu,k),r)\subseteq C(\mu,r)$. In particular we get that
$C(C(S,k),r)\subseteq C(S,r)$ for every nonnegative matrix $S$.
\end{corollary}
\begin{proof}
The result follows from the continuity properties of $C(\cdot,r)$ on the space of $k\times k$ matrices.
\end{proof}

\begin{lemma}\label{ketthat}
Let $\{T_i\}_{i=1}^\infty$ be a sequence of nonnegative symmetric consistent tables
and let $\mu$ be the symmetric measure constructed as in Lemma~\ref{tabcon}. Then for every natural number $k$, the set $C(\mu,k)$ is the topological closure of
$$\bigcup_{n\in\mathbb{N}}C(T_n,k).$$
\end{lemma}
\begin{proof}
For a node $x\in V(B)$ in the binary tree $B$, let $\chi_x$ be the
characteristic function of the set $S_x$. Let
$h:V(B)\rightarrow\mathbb{N}$ correspond to the
levels of the nodes of $B$. The system of functions
$$\{\chi_x~|~h(x)=n\}$$ show that $T_n\in C_0(\mu,2^n)$. Then
Lemma~\ref{resze} shows that $C(T_n,k)\subseteq C(\mu,k)$ for all
$n$ and $k$.

For the other direction, we shall prove that every element $X\in C_c(\mu,k)$ can be
approximated by elements from $C(T_n,k)$. Let $f_1,f_2,\dots,f_k$ be continuous
functions on $\mathcal{C}$ such that $\sum_{i=1}^{k}f_i=1$,
$\int_{\mathcal{C}}f_i~d\nu=1/k$ and
$X=\mathcal{M}(f_1,f_2,\ldots,f_k)$.
Since $\mathcal{C}$ is compact, these functions are actually uniformly continuous, and hence we can for a given $\varepsilon>0$ find an $n\in\mathbb{N}^+$ and functions $g_1,\ldots,g_k$ with values in $[0,1]$ such that each $g_i$ is constant on sets $S_x$ with $h(x)=n$, and
\[
\|f_i-g_i\|_\infty<\varepsilon.
\]
Let now $Y\in C(T_n,k)$ be generated by the family $(g_1,\ldots,g_k)$.
Then we clearly have 
\[
\|X-Y\|_\infty\leq (2\varepsilon+\varepsilon^2)\cdot\mu(\mathcal{C}\times\mathcal{C}).
\]

\end{proof}

\begin{lemma}\label{itav}
Let $c$ be a positive number and let $k$ and $t$ be natural numbers.
Then for any nonnegative $n\times n$ matrix $S$ with $\gamma(S)=c$, we have
$$\dist(C(C(S,t),k),C(S,k))\leq\frac{2ck}{t}$$
\end{lemma}
\begin{proof}
Let $X$ be an arbitrary element of $C(S,k)$ and let
$M\in\mathcal{K}(k,n)$ be a matrix with $MSM^T=X$. Let $mk+r=t$ for
some natural numbers $m$ and $r<k$. Now we construct a $t\times n$
matrix $M_1$ from $M$ with the following operations:

\noindent 1.) we multiply $M$ by $k/t$

\noindent 2.) we replace each row $v$ by an $m\times k$ matrix whose
rows are all equal to $v$.

\noindent 3.) we add $r$ copies of the everywhere $1/t$ row of
length $n$. (These are the last rows of the resulting matrix.)

The matrix $M_1$ is an element of $\mathcal{K}(k,n)$ and thus
$Y=M_1SM_1^T\in C(S,t)$.

Let $M_2$ be a $k\times t$ matrix described as follows. The $i$-th
row of $M_2$ has a block of $1$-s of length $m$ starting at the
$im+1$-st place, it has a block of $1/k$ of length $r$ starting at
the $mk+1$-st place and every other place is $0$. It is clear that
$M_2\in\mathcal{K}(k,t)$.

For technical reasons we construct a $t\times n$ matrix $M_3$ by
putting zeroes to the last $r$ rows of $M_1$. Let
$Z=M_2M_1SM_1^TM_2^T=M_2YM_2^T\in C(C(S,t),k)$. From $M_3<M_1$ we
obtain
$$\frac{mk}{t}X=M_2M_3SM_3^TM_2^T\leq Z.$$ Using that
$\gamma(\frac{mk}{t}X)=\frac{cmk}{t}$ and that $\gamma(Z)=c$ we
get that dist$(Z,\frac{mk}{t}X)\leq \frac{cr}{t}$ and thus $\dist(Z,X)\leq \frac{2cr}{t}<\frac{2ck}{t}$. Finally, Lemma \ref{tart} concludes the proof.
\end{proof}

\begin{lemma}\label{protav}
Let $X$ and $Y$ be two nonnegative matrices of the same size $n$
such that $\|X-Y\|_1\leq\varepsilon$ and let
$M\in\mathcal{K}(k,n)$ for some natural number $k$. Then $$\|MXM^T-MYM^T\|_1\leq\varepsilon.$$
\end{lemma}

\begin{proof}
Let us express $X-Y$ as $A-B$ with two nonnegative matrices $A,B$
where $A$ is the positive part of $X-Y$ and $B$ is the negative
part. Now $\|X-Y\|_1=\gamma(A)+\gamma(B)\leq \varepsilon$. It
follows that
$$\|MXM^T-MYM^T\|_1=\|M(A-B)M^T\|_1=\|MAM^T-MBM^T\|_1\leq$$
$$\leq\|MAM^T\|_1+\|MBM^T\|_1=\gamma(MAM^T)+\gamma(MBM^T)=\varepsilon.$$
\end{proof}

The previous lemma has the following immediate corollary:

\begin{lemma}\label{mattav}
Let $X$ and $Y$ be two nonnegative matrices of the same size $n$
such that $\|X-Y\|_1\leq\varepsilon$. Then $\mathrm{
dist}(C(X,k),C(Y,k))\leq\varepsilon$ for every natural number $k$.
\end{lemma}

\begin{lemma}\label{hartav}
Let $X$ and $Y$ be two nonnegative matrices such that
$\gamma(X),\gamma(Y)\leq c$ and $\dist(C(X,n),C(Y,n))\leq\varepsilon.$ Then $\dist(C(X,k),C(Y,k))\leq\varepsilon+4ck/n$.
\end{lemma}
\begin{proof}
Using Lemma~\ref{mattav}, from $\dist(C(X,n),C(Y,n))\leq\varepsilon$, we get that $$\dist(C(C(X,n),k),C(C(Y,n),k))\leq\varepsilon.$$ Now Lemma~\ref{itav}
completes the proof.
\end{proof}

\begin{lemma}\label{apcon}
Let $k_1<k_2<\dots$ be an arbitrary strictly increasing infinite
sequence of natural numbers and let $S_1,S_2,\ldots$ be a convergent
sequence of matrices with $\gamma(S_i)\leq c$. Let furthermore
$X_1,X_2,\dots$ be a sequence of matrices with $\gamma(X_i)\leq c$
such that
$$\dist(C(X_i,k_i),\lim_{n\rightarrow\infty}C(S_n,k_i))\leq 1/i$$ for
every natural number $i$. Then the sequence $X_i$ is convergent with
$$\lim_{n\rightarrow\infty}C(X_n,k)=\lim_{n\rightarrow\infty}C(S_n,k)$$ for every natural
number $k$.
\end{lemma}
\begin{proof}
From the above conditions we get that for each index $i$ there is an
index $n_i$ such that
$$\dist(C(X_i,k_i),C(S_n,k_i))\leq 2/i$$
for all $n>n_i$. From Lemma~\ref{hartav} it follows that
$$\dist(C(X_i,k),C(S_n,k))\leq 2/i+4ck/k_i$$ if $n>n_i$.
This means that
$$\dist(C(X_i,k),\lim_{n\rightarrow\infty}C(S_n,k))\leq
2/i+4ck/k_i.$$ Consequently
$$\lim_{i\rightarrow\infty}C(X_i,k)=\lim_{n\rightarrow\infty}C(S_n,k).$$
\end{proof}

\section{A regularity lemma for measures}\label{ch:reg}

\begin{lemma}\label{balancing}
Let $c,\varepsilon$ be positive numbers and let $k_1, k_2\ldots, k_r$ be positive integers for some $r\in\mathbb{N}^+$. Then there exists a positive integer 
\[
m<\max\left(\frac{4c\left(\prod_{i=1}^{r}
k_i\right)^2}{\varepsilon},\prod_{i=1}^{r} k_i\right)
\]
that 
can be expressed as $(\prod_{i=1}^{r}
k_i)^{\alpha}$ for some natural number $\alpha$ with the following property.

Let $S$ be a nonnegative $n\times n$ matrix with $\gamma(S)\leq c$. For any system  $\{X_1,X_2,\dots,X_r\}$ of matrices with $X_i\in C(S,k_i)$ for all $1\leq i\leq
r$ and some $k_i\in\mathbb{N}$ we can find a matrix $Z\in C(S,m)$ such that
for every $1\leq i\leq r$, there exists a matrix
$\mathcal{M}_i\in\widehat{\mathcal{K}}(k_i,m)$ with
\[
\left\|\mathcal{M}_iZ\mathcal{M}_i^T-X_i\right\|_1<\varepsilon.
\]
\end{lemma}
\begin{proof}
For all $1\leq i\leq r$, let $M_i$ be an element of
$\mathcal{K}(k_i,n)$ such that $M_iSM_i^T=X_i$.\\
The initial idea is to create some sort of a joint ``blow-up'' $M$ of these matrices that can be used to generate $Z$. In other words, $M$ should have column sums all equal to 1, and all row sums also equal. At the same time, we would need $M$ to have the property that grouping and adding up its rows in different (equitable) ways, we obtain each of the $M_i$-s.\\
To deal with our second goal, we first create a
$k_1k_2\ldots k_r\times n$ matrix $T$ whose rows are indexed by
sequences $(a_1,a_2,\dots,a_r)$ with $1\leq a_i\leq k_i$ such that
the $j$-th element of the row corresponding to every such sequence
is $\prod_{i=1}^{r}M_i(a_i,j)$. Note that given $1\leq j\leq r$, adding up the rows with the same corresponding $a_j$ yields $M_j$. We denote by $g(a_1,\dots,a_r)$
the sum of the elements in the row corresponding to
$(a_1,a_2,\dots,a_r)$.\\
The issue with this matrix is that although it is a blow-up of the $M_i$-s, and all its column sums are 1, the row sums do not match up. To correct this, we want to further blow-up each row according to its weight $g(\cdot)$. However, the row sums need not have ratios that allow for this, and so we shall scale down $T$, do some rounding in the number of blow-ups of each row, and use uniform rows to compensate. The actual error in approximating the $X_i$-s will come from this step.\\
Let $\vec{1}$ denote the all $1$ vector of
length $n$. We introduce a new matrix $T_{p,q}$ for two natural
numbers $p,q$ with $p+k_1k_2\dots k_r\leq q$, which is obtained from
$T$ by the following operations:

\begin{itemize}
\item  For every possible sequence $(a_1,a_2,\dots,a_r)$ we
replace the corresponding row $\vec{w}$ of $T$ by
$$\frac{p}{q\lceil g(a_1,a_2,\dots,a_r)p/n \rceil}\vec{w} +\left(\frac{1}{q}-\frac{pg(a_1,a_2,\dots,a_r)}{qn\lceil g(a_1,a_2,\dots,a_r)p/n
\rceil}\right) \vec{1} .$$ (If $g(a_1,a_2,\dots,a_r)=0$, then we replace $\vec{w}$ by
$\frac{1}{q} \vec{1}$.)

\item Next, for every possible sequence $(a_1,a_2,\dots,a_r)$, we replace the corresponding row by $\lceil
g(a_1,a_2,\dots,a_r)p/n\rceil$ identical copies of the same row. Let
us denote the resulting matrix by $\widehat{T}_{p,q}$.

\item Next, we add $q-\ell$ copies of $\frac{1}{q}\vec{1}$ as new rows where $\ell$
denotes the number of rows in $\widehat{T}_{p,q}$. (We shall show below that
$q-\ell\geq 0$.)
\end{itemize}

Since $$\sum_{1\leq a_i\leq k_i} g(a_1,a_2,\dots,a_r)=n,$$ we
get that the number of rows in $\widehat{T}_{p,q}$ is $$\sum_{1\leq
a_i\leq k_i} \lceil g(a_1,a_2,\dots,a_r)p/n\rceil$$ which is at most
$p+k_1k_2\dots k_r$. This means that $\widehat{T}_{p,q}$ has indeed at most $q$
rows. The matrix $T_{p,q}$ has $q$ rows, each with element sum
$n/q$. Furthermore a simple calculation shows that $T_{p,q}$ is
a nonnegative matrix and in each column the sum of the elements is
$1$. Consequently $T_{p,q}\in\mathcal{K}(q,n)$.

Let us fix an index $1\leq j \leq r$ and $t \in [k_j]$. Let $F_t$ denote the collection
of those rows in $\widehat{T}_{p,q}$ which are coming from a row in $T$
indexed by a sequence $(a_1,\ldots,a_r)$ with $a_j=t$. The equation
$$\sum_{a_e=f, 1\leq a_i\leq k_i} g(a_1,a_2,\dots,a_r)=n/k_e$$
shows that $p/k_j\leq |F_t|\leq p/k_j+(k_1k_2\dots k_r)/k_j\leq
q/k_j$. This means that in $T_{p,q}$ we can enlarge the row sets
$F_t$ (where $1\leq t\leq k_j$) in such a way that we obtain a balanced partition
$\Omega_j$ of the row set of $T_{p,q}$ into $k_i$ sets.

Let $Y_j$ be the $k_j\times q$ matrix which is the characteristic
matrix of the previous partition. Let $\bar{T}_{p,q}$ be the matrix
constructed from $T$ the same way as $T_{p,q}$ but using $\vec{0}$ instead of $\vec{1}$ everywhere. It is
easy to see that for this matrix
$$Y_j \bar{T}_{p,q}=\frac{p}{q} M_j.$$
Let $Z=T_{p,q}ST_{p,q}^T\in C(S,q)$. We have that
$$\frac{p^2}{q^2}X_j = Y_j \bar{T}_{p,q} S \bar{T}_{p,q}^T Y_j^T \leq
Y_j T_{p,q} S T_{p,q}^T Y_j^T=Y_jZY_j^T.$$
If $\gamma(S)=c_1\leq c$ then
$\gamma(\frac{p^2}{q^2}X_j)=c_1\frac{p^2}{q^2}X_j$ and
$\gamma(Y_jZY_j^T)=c_1$. It follows that
$$\left\|\frac{p^2}{q^2}X_j-Y_jZY_j^T\right\|_1\leq c_1\left(1-\frac{p^2}{q^2}\right).$$
From $\|X_j-\frac{p^2}{q^2}X_j\|_1\leq c_1\left(1-\frac{p^2}{q^2}\right)$ we obtain
that
$$\left\|X_j-Y_jZY_j^T\right\|_1\leq 2c_1\left(1-\frac{p^2}{q^2}\right)\leq 4c_1\left(1-\frac{p}{q}\right).$$
Let $d=k_1k_2\dots k_r$ and $q_1$ be the largest integer power of
the product $d$ which is not larger than $4cd^2/\varepsilon$ and let
$m:=q=\max(q_1,d)$. Let $p=q-d$. Since $q>4cd/\varepsilon$, we obtain that
$4c_1\left(1-\frac{p}{q}\right)<\varepsilon$, meaning that $Z$ satisfies the
required properties.
\end{proof}

\begin{remark}
Note that the above choice of $m$ lets us improve the approximation by a constant, i.e., instead of the upper bound $<\varepsilon$, we actually have $\leq 4cd/m$.
\end{remark}

The same statement as in the previous lemma holds if we replace $S$
by a convergent sequence $S_i$. More precisely, we have the following result.

\begin{corollary}\label{limbal}
Let $S_i$ be a shape-convergent sequence of nonnegative matrices with $\gamma(S_i)\leq c$, and for each $k\in\mathbb{N}^+$ let $C(k):=\lim C(S_i,k)$.
Let $\varepsilon$ be a positive number, let $k_1,k_2,\ldots,k_r$ be positive integers, and let $m$ be the positive integer yielded by Lemma \ref{balancing}. Then for any system $\{X_1,X_2,\dots,X_r\}$
of matrices with $X_i\in C(k_i)$ for all $1\leq i\leq
r$, there exists a matrix $Z\in C(m)$ such that
for every $1\leq i\leq r$, there is a matrix
$\mathcal{M}_i\in\widehat{\mathcal{K}}(k_i,m)$ with
\[
\left\|\mathcal{M}_iZ\mathcal{M}_i^T-X_i\right\|_1<\varepsilon.
\]

\end{corollary}
\begin{proof}
Let again $d:=k_1\ldots k_r$, and recall that $m>4cd/\varepsilon$. Thus we may choose a positive number 
\[
\delta<\frac{1}{2}(\varepsilon-4cd/m).
\] 
Now choose an $S_\ell$ such that $d(C(S_\ell,k),C(k))\leq\delta$ for all $1\leq k\leq m$. Then we can find matrices $\widetilde{X}_i\in C(S_\ell,k_i)$ with $\|\widetilde{X}_i-X_i\|_1\leq\delta$ ($1\leq i\leq r$), and apply the previous lemma to obtain matrices $\widetilde Z\in C(S_\ell,m)$ and $\mathcal{M}_i\in\widehat{\mathcal{K}}(k_i,m)$ with
\[
\left\|\mathcal{M}_i\widetilde{Z}\mathcal{M}_i^T-\widetilde{X}_i\right\|_1<\varepsilon
\]
for each $1\leq i\leq r$.
Now choose a $Z\in C(m)$ with $\|\widetilde{Z}-Z\|_1\leq\delta$. Then by Lemma \ref{protav} we obtain for each $1\leq i \leq r$ that
\begin{align*}
\left\|\mathcal{M}_iZ\mathcal{M}_i^T-X_i\right\|_1
&\leq
\left\|\mathcal{M}_iZ\mathcal{M}_i^T-\mathcal{M}_i\widetilde{Z}\mathcal{M}_i^T\right\|_1+\left\|\mathcal{M}_i\widetilde{Z}\mathcal{M}_i^T-\widetilde{X}_i\right\|_1+\left\|\widetilde{X}_i-X_i\right\|_1\\
&\leq \delta+4cd/m+\delta<\varepsilon.
\end{align*}
\end{proof}

\begin{lemma}\label{pointset}
Let $S$ be a nonnegative matrix with $\gamma(S)=c$. Then for every
natural number $k$ and positive number $\varepsilon$ there is a set
$H\subset C(S,k)$ such that

\noindent 1.) $\dist(H,S(S,k))\leq \varepsilon$

\noindent 2.) $|H|\leq (\lceil ck^2/\varepsilon\rceil)^{k^2}.$
\end{lemma}

\begin{proof}
Since every element $X$ of $C(S,k)$ is a nonnegative matrix with
$\gamma(X)=c$ we have that $C(S,k)$ is contained in the cube
$C=[0,c]^{k^2}$. Let us divide $C$ into small cubes of side length
$\varepsilon/k^2$. We construct $H$ such that from each small cube,
which intersects $C(S,k)$ we choose a point which is contained in
the intersection. Since the diameters of the small cubes (measured
in $\ell_1$) are $\varepsilon$ the set $H$ satisfies the required
properties.
\end{proof}

\begin{proposition}\label{regularity}
Let $S$ be a nonnegative matrix with $\gamma(S)=c$, let $k$ be a
natural number and let $\varepsilon$ be a positive number. Then there
is a natural number $$m<(4c/\varepsilon)k^{(2ck^2/\varepsilon)^{k^2}}$$
which is a power of $k$ and an element $X\in C(S,m)$ such that $\dist(C(X,k),C(S,k))\leq\varepsilon$.
\end{proposition}
\begin{proof}
We apply Lemma \ref{pointset} for $S$ with $\varepsilon/2$ and then
Lemma \ref{balancing} for the resulting set $H$ with $\varepsilon/2$.
\end{proof}

\section{Proof of Theorem~\ref{thm:shapeMeasure} and Theorem~\ref{thm:shpm2b}}\label{ch:main} 

Throughout this section $S_1,S_2,\dots$ is a convergent sequence of
nonnegative matrices. Using the convergence, we can assume that
$\gamma(S_i)<c$ for all $i$. We denote by $C(k)$ the limit
$\lim_{i\rightarrow\infty} C(S_i,k)$ which is a closed subset of
$[0,c]^{k^2}$. Using Corollary~\ref{tart} and Lemma~\ref{mattav}, it
follows that for an arbitrary $r$ we have that $C(C(r),k)\subseteq
C(k)$.

\begin{lemma}\label{limreg} For every positive number $\varepsilon$ and natural
number $k$, there is a natural number $q$ which is a power of $k$ and
an element $X\in C(q)$ such that
$$\dist(C(X,k),C(k))\leq\varepsilon$$.
\end{lemma}
\begin{proof}
By Proposition~\ref{regularity}, we have a sequence of natural
numbers $m_i$ and $X_i\in C(S_i,m_i)$ such that $\dist(C(X_i,k),C(S_i,k))\leq\varepsilon$ where the numbers $m_i$ are
all powers of $k$ and furthermore they are all smaller than a value
depending only on $c,k$ and $\varepsilon$. This means that there is
only finitely many possible values for the numbers $m_i$ and thus
there is an infinite sequence $t_1<t_2<\dots$ of natural numbers
with $m_{t_i}=r$ for some fixed number $r$ and all $i$. Now the
elements $X_{t_i}$ are all in $[0,c]^{r^2}$ and so there is a
convergent subsequence $X_{a_i}$ whose limit is $X$. Lemma
\ref{mattav} implies that $X$ satisfies the required property.
\end{proof}

Lemma~\ref{limreg} shows that there is a sequence of elements
$X_1,X_2,\dots$ such that $X_i\in C(2^{t_i})$ and $\dist(
C(X_i,2^{i}),C(2^i))\leq 1/2i$. We create another sequence
$Y_1,Y_2,\dots$ of matrices recursively. Let $Y_1=X_1$ and $Y_i\in
C(2^{g_i})$ be a matrix constructed in Lemma~\ref{limbal}
for $Y_{i-1}$ and $X_i$ with $\varepsilon=1/10^i$. Let
$M_i\widehat{\mathcal{K}}(2^{g_{i-1}},2^{g_i})$ be the sequence of
balanced partition matrices (guaranteed by Lemma~\ref{limbal})
for which $\|M_iY_iM_i^T-Y_{i-1}\|_1\leq 1/10^i$. For two natural
numbers $i,j$ with $i>j$ we introduce matrices $Y_{i,j}$ defined by
$$Y_{i,j}=M_{j+1}M_{j+2}\dots M_{i}Y_iM_{i}^T\dots
M_{j+2}^TM_{j+1}^T.$$ An iterated use of Lemma \ref{protav} implies
that
$$\|Y_{i,j}-Y_j\|_1\leq 10^{-j}+10^{-j+1}+\dots+10^{-i-1}<2\cdot 10^{-i-1}.$$
Now we can choose an increasing sequence $d_i$ from the natural
numbers such that $Y_{d_i,j}$ converges to some $\bar{Y}_j$ for all
fixed number $j$. We have that
$$\|\widehat{Y}_{j}-Y_j\|_1< 2\cdot 10^{-j-1}$$
and that for $i>j$
$$\bar{Y}_{j}=M_{j+1}M_{j+2}\dots M_{i}\bar{Y}_iM_{i}^T\dots
M_{j+2}^TM_{j+1}^T.$$ From the construction of the sequence $Y_i$ it
follows that there is a matrix
$N_i\in\widehat{\mathcal{K}}(2^{t_i},2^{g_i})$ with
$\|N_iY_iN_i^T-X_i\|_1< 10^{-i}$. By Lemma~\ref{protav} we get that
$$\|N_i\bar{Y}_iN_i^T-X_i\|_1< 3\cdot 10^{-i}.$$ Using that
$C(\bar{Y}_i,2^i)\supseteq C(N_i\bar{Y}_iN_i^T,2^i)$ and Lemma
\ref{mattav} we obtain that
$$\dist(C(\bar{Y}_i,2^i),C(2^i))< 1/2i+3\cdot 10^{-i}<1/i.$$
Applying Lemma \ref{apcon} it follows that
$$\lim_{n\rightarrow\infty}C(\bar{Y}_n,k)=C(k)$$
for all natural numbers $k$.

Now the construction of the sequence $\bar{Y}_i$ shows that it can
be refined to a consistent sequence of matrices $\bar{T}_i\in
C(2^i)$. Corollary~\ref{tart}~and Lemma~\ref{ketthat} complete the
proof.

\medskip

Now we enter the proof of Theorem \ref{thm:shpm2b}. Let $\mu$ be an s-graphon represented on the cantor set $\mathcal{C}$. Using that $\mu$ arises from a sequence of consistent tables and Lemma \ref{ketthat} we have that $\mu$ is the limit of non-negative symmetric matrices $\{S_i\}_{i=1}^\infty$ with $\gamma(S_i)=1$. Each matrix $S_i$ can be represented by an s-graphon $Q_i$ whose Radon-Nikodym derivative $W_i$ is is a step function. We know from dense graph limit theory \cite{LSz} that $W_i/\|W_i\|_\infty$ is the limit (according to the limit notion for dense graph sequences) of some graph sequence $\{G_j^i\}_{j=1}^\infty$. This implies that a constant multiple of $Q_i$ is the limit of $\{G_j^i\}_{j=1}^\infty$ according to s-convergence. This constant must be $1$ because $Q_i$ is already a probability measure. Using that $\mu$ is the limit of $\{Q_i\}_{i=1}^\infty$ the proof is complete.

\section{Cantor set versus $[0,1]$}\label{ch:cantor}

\begin{Le}\label{le:cantor-lebesgue}
There exists an invertible measure preserving map $\Phi:\mc{C}\to [0,1]$ between the measure spaces $(\mc{C},\ms{B}(\mc{C}),\nu)$ and $([0,1],\ms{B}([0,1]),\lambda)$.
\end{Le}
\begin{proof}
Let $\psi:\mc{C}\to [0,1]$ be the standard representation of the Cantor set on $[0,1]$, and denote by $B_0\in [0,1]$ the countable set of points whose ternary representation is eventually constant 0 or 2 (in particular $B_0\subset\psi(\mc{C})$). Clearly $\psi$ is an invertible measure preserving map between $\mc{C}$ and $\psi(\mc{C})$, where the latter is equipped with the push-forward measure $\nu^\pi$. Now consider the surjective map $\varphi_0:\psi(\mc{C})\to [0,1]$ given through $\varphi_0(x):=\nu^\pi([0,x]\cap\psi(\mc{C}))$. It can easily be seen that the restriction $\varphi_0|_{\psi(\mc{C})\backslash B_0}$ is a measure preserving bijection.\\
Then the map $\widehat{\Phi}:=\varphi_0|_{\psi(\mc{C})\backslash B_0} \circ \psi^{-1}(\psi(\mc{C})\backslash B_0)$ will be an invertible measure preserving map between $\mc{C}\backslash B$ and $[0,1]\backslash A$, where $B:=\psi^{-1}(B_0)$ and $A:=\varphi_0(B_0)$. Both of the sets $A$ and $B$ are countably infinite, and hence $\widehat{\Phi}$ may be extended to an appropriate $\Phi:\mc{C}\to [0,1]$.
\end{proof}

The above lemma has the following useful consequences.

Through the invertible map $\Phi$ of Lemma \ref{le:cantor-lebesgue} there is a one-to-one correspondence between graphons defined on $[0,1]^2$ and graphons defined on $\mc{C}^2$. Indeed, given a graphon $W:[0,1]^2\to\mb{R}$, we obtain an (almost everywhere defined) graphon $(W\circ\Phi)(x,y):=W(\Phi(x),\Phi(y))$ on $\mc{C}^2$, and vice-versa using $\Phi^{-1}$.
This correspondence preserves the $L^p$-norm of the graphons for all $1\leq p\leq\infty$, and also leaves the cut norm and cut distance invariant.\\
In particular, for any two graphons $W_1,W_2$ on $\mc{C}^2$, we have that
\begin{eqnarray*}
&&\delta_\square(W_1,W_2)\\
&=&\inf\left\{\left.\|W_1-W_2\circ\varphi\|_\square\right|
\varphi:\mc{C}\to\mc{C} \mbox{ is invertible and measure preserving }\right\}
\end{eqnarray*}
and
\begin{eqnarray*}
\|W_1\|_\square&=&\max\left|\left\{\left.\int_{\mc{C}^2} W_1(x,y)f(x)g(y) d\nu\times\nu
\right|
f,g:\mc{C}\to[0,1]\right\}\right|\\
&=&\max\left|\left\{\left.\int_{S\times T} W_1(x,y) d\nu\times\nu
\right|
S,T\in\ms{B}(\mc{C})\right\}\right|
\end{eqnarray*}

\section{$\alpha$-shapes}\label{ch:ash}

\begin{D}
Let $\ms{L}^\infty(\mc{C})$ denote the vector space of bounded measurable functions on $\mc{C}$. For a positive integer $k$, let 
\[
\ms{F}_k^{(1)}:=\left\{ f\in \ms{L}^\infty(\mc{C})^k\left| \sum_{i=1}^k f_i(x)=1 \Forall x\in\mc{C}\right.\right\}.
\]
Given a positive real vector $\alpha\in\mb{R}^k$, let
\[
\mk{F}_\alpha:=\left\{
f\in \ms{F}^{(1)}_k
\left|
\int_{\mc{C}}f_j 
\mr{d}
\lambda=\alpha_j
\mbox{ and }
f_j\geq0  \Forall 1\leq j\leq k
\right.
\right\}.
\]
As a short-hand notation, we shall also use
\[
\mk{F}_k:=\left\{
f\in \ms{F}_k^{(1)}
\left|
\int_{\mc{C}}f_j 
\mr{d}
\lambda=1/k
\mbox{ and }
f_j\geq0 \Forall 1\leq j\leq k
\right.
\right\}
\]
\end{D}

Let $\mc{M}^+_c$ be the space of symmetric Borel measures on $\mc{C}\times\mc{C}$ with $\mu(\mc{C}\times\mc{C})\leq c$. Each symmetric matrix $S$ with $\gamma(S)\leq c$ and $S_{ij}\geq 0$ naturally induces a measure on $[0,1]^2$, and thus through $\Phi^{-1}$ an element $\mu_S\in\mc{M}^+_c$, and we clearly have $C(S,k)=C(\mu_S,k)$. Also, each non-negative symmetric function $W\in L^1(\mc{C}\times\mc{C},\mb{R})$ with $\|W\|_1\leq c$ induces a measure $\mu_W\in\mc{M}^+_c$, and we shall define
$C(W,k):=C(\mu_W,k)$.\\
Further, given a positive integer $k$, let $E_k:=\left\{\alpha\in[0,1]^k\left|\sum_{i=1}^k \alpha_i=1\right.\right\}$. Given a closed set $A\subset [0,1]^k$ with $A\cap E_k\neq\emptyset$, let $\widetilde{C}_0(\mu,A)\subset \mb{R}^{k+k^2}$ denote the set
\[
\left\{
(\alpha,M)\in (A\cap E_k)\times\mb{R}^{k\times k}
\left|
\Exists f\in\mk{F}_\alpha\colon
M_{i,j}=\int f_i\otimes f_j\mr{d}\mu
\right.
\right\},
\]
and $\widetilde{C}(\mu,A)\subset \mb{R}^{k+k^2}$ its closure.

Note that with a small abuse of notation, we shall define
\begin{align*}
\widetilde{C}_0(\mu,\alpha)&:=\left\{
M\in\mb{R}^{k\times k}
\left|
(\alpha,M)\in\widetilde{C}_0(\mu,\{\alpha\})
\right.
\right\},\\
\widetilde{C}(\mu,\alpha)&:=\left\{
M\in\mb{R}^{k\times k}
\left|
(\alpha,M)\in\widetilde{C}(\mu,\{\alpha\})
\right.
\right\}
\end{align*}
for $\alpha\in[0,1]^k\cap E_k$. With this notation, we have
\[
C(\mu,k)=\widetilde{C}\left(\mu,\frac{\vec{1}}{k}\right).
\]

\begin{Le}\label{Le:unif_equi}
For each $m\in\mb{N}^+$ and $\varepsilon\in(0,1/m]$, the functions 
\[
\widetilde{C}(\mu,\cdot):(\varepsilon,1]^m\cap E_m\to\mc{K}(\mb{R}^{m\times m})
\]
given by $\alpha\mapsto \widetilde{C}(\mu,\alpha)$ are uniformly equicontinuous on $\mc{M}^+_c$. In fact, they are Lipschitz continuous with constant $2c(m-1)/\varepsilon$.
\end{Le}

\begin{proof}
Fix an $\eta\in (0,\varepsilon/m)$ and consider $\alpha,\beta\in[\varepsilon,1]^m\cap E_m$ such that $\|\alpha-\beta\|_\infty\leq\eta$. Given any $f\in\mk{F}_\alpha$, there always exists a $g\in\mk{F}_\beta$ such that for each $1\leq i\leq m$ we have $\|f_i-g_i\|_\infty\leq(m-1)\eta/\varepsilon$. Then we obtain for all $1\leq i,j\leq m$ and $\mu\in\mc{M}^+_c$ that
\begin{align*}
\left|
M(\mu,f)_{i,j}-M(\mu,g)_{i,j}
\right|&=
\left|
\int f_i\otimes f_j\mr{d}\mu-\int g_i\otimes g_j\mr{d}\mu
\right|\\
&\leq
\|f_i\otimes f_j-g_i\otimes g_j\|_\infty\cdot\|\mu\|\leq 2(m-1)\eta/\varepsilon \cdot c.
\end{align*}
Consequently, $\widetilde{C}(\mu,\alpha)$ and $\widetilde{C}(\mu,\beta)$ have an $\ell_\infty$-Hausdorff distance of at most $2c(m-1)\eta/\varepsilon$ whenever $\|\alpha-\beta\|_\infty\leq\eta$, independently of $\mu$, and the claimed Lipschitz continuity follows.
\end{proof}

As an easy consequence of this uniform Lipschitz continuity, we have the following result.
\begin{Cor}\label{Cor:cont}
For any sequence $(\mu_n)_{\mb{N}^+}\subset \mc{M}^+_c$, size $k\in\mb{N}^+$ and compact set $A\subset(0,1]^k$, the convergence of $\widetilde{C}(\mu_n,A)$ implies the convergence of any section sequence $\widetilde{C}(\mu_n,\alpha)$ where $\alpha\in A\cap E_k$.
\end{Cor}

The above results exclude the cases when any coordinate of $\alpha$ is zero, so our next aim is to investigate what results still hold if we do allow for degenerate families of functions.
Defining the set
\[
\widehat{C}(\mu,k):=\overline
{\bigcup_{\varepsilon\in\left(0,\frac{1}{k}\right]} \widetilde{C}(\mu,[\varepsilon,1]^k)
},
\]
the question is how this compact set relates to $\widetilde{C}(\mu, [0,1]^k)$ and to the sets $\left(\beta,\widetilde{C}(\mu,\beta)\right)$ where some coordinate of $\beta$ is zero.

\begin{A}\label{Prop:zero}
For any $\mu\in\mc{M}^+_c$, we have
\[
\widetilde{C}(\mu, [0,1]^k)= \widehat{C}(\mu,k).
\]
If $\mu$ is absolutely continuous with respect to $\nu\times\nu$, we actually have
\[
\widetilde{C}(\mu, [0,1]^k)=\bigcup_{\alpha\in[0,1]^k, \sum_{i=1}^k\alpha_i=1}\left(\alpha,\widetilde{C}(\mu,\alpha)\right).
\]
\end{A}
\begin{proof}
For the first part it is enough to show that if $\beta\in[0,1]^k$, $\sum_{i=1}^k \beta_i=1$, and $\beta_\ell=0$ for some $1\leq \ell\leq k$, then for any family of functions $f\in\mk{F}_\beta$ and small enough $\varepsilon>0$ we can find a family of functions $g\in\mk{F}_\alpha$ with $\alpha\in(0,1]^k$ and $\sum_{i=1}^k \alpha_i=1$ such that
\[
\|M(\mu,f)_{i,j}-M(\mu,g)_{i,j}\|_\infty<c\varepsilon.
\]
Suppose $\varepsilon<\min(1,c/4)$ and define the functions
\[
g_i:=\left(1-\frac{\varepsilon^2}{2c}\right)f_i+\frac{\varepsilon^2}{2ck}.
\]
Clearly each $g_i$ is positive, their sum is the constant $1$ function, and thus they form a family $g\in\mk{F}_\alpha$ for some appropriate $\alpha$ with no zero entries. It remains to be shown that the corresponding matrix is close enough to the one pertaining to the family $f$.

Given indices $1\leq i,j\leq k$ we have
\begin{align*}
&\left|
M(\mu,f)_{i,j}-M(\mu,g)_{i,j}
\right|=
\left|
\int f_i\otimes f_j\mr{d}\mu-\int g_i\otimes g_j\mr{d}\mu
\right|\\
={}&
\left|
\int
\left(
f_i\otimes f_j-
\left(
\left(1-\frac{\varepsilon^2}{2c}\right)f_i+\frac{\varepsilon^2}{2ck}
\right)
\otimes
\left(
\left(1-\frac{\varepsilon^2}{2c}\right)f_j+\frac{\varepsilon^2}{2ck}
\right)
\right)
\mr{d}\mu
\right|\\
={}&
\left|
\int
\left(
\left(\frac{\varepsilon^2}{c}-
\frac{\varepsilon^4}{4c^2}\right)(f_i\otimes f_j)-
\left(1-\frac{\varepsilon^2}{2c}\right)\frac{\varepsilon^2}{2ck}
\left(f_i\otimes 1\right)\right.\right.\\
&-
\left.\left.
\left(1-\frac{\varepsilon^2}{2c}\right)\frac{\varepsilon^2}{2ck}
(1\otimes f_j)
+\frac{\varepsilon^4}{4c^2k^2}
1\otimes 1
\right)
\mr{d}\mu
\right|\leq \varepsilon^2+\frac{2\varepsilon^2}{2k}+\frac{\varepsilon^4}{4ck^2}<
c\varepsilon,
\end{align*}
and we are done.\\
For the second part, we wish to show that if $\mu$ is absolutely continuous with respect to $\nu\times\nu$, then for any $\beta\in[0,1]^k\cap E_k$ with $\beta_\ell=0$ for some $1\leq \ell\leq k$, whenever $(\beta,M)\in \widehat{C}(\mu,k)$, then also $M\in\widetilde{C}(\mu,\beta)$.
Fix an $\varepsilon\in(0,1)$, and consider the set $B:=\left\{1\leq \ell\leq k|\beta_\ell=0\right\}$.
Then we may find an $\alpha\in(0,1]^k\cap E_k$ and
a family $f\in\mk{F}_\alpha$ such that
 $\|M-M(\mu,f)\|_\infty\leq \varepsilon$ and $0<\alpha_j(1-\varepsilon/ck)<\beta_j$ for all $j\notin B$.\\
For $j\notin B$ let now $x_j:=\max\{0,(\alpha_j-\beta_j)/\alpha_j\}<\varepsilon/ck$ and $y_j:=\max\{0,\beta_j-\alpha_j\}$.
Consider the following family of functions.
\begin{align*}
g_i &:=0 & \mbox{ if } i\in B;\\
g_i &:=(1-x_i) f_i+\frac{y_i}{\sum_{\ell\notin B}y_\ell}\sum_{b\in B} x_b f_b
 & \mbox{ if } i\notin B.
\end{align*}
Then we have $g\in\mk{F}_\beta$ and it remains to be shown that $M(\mu,g)$ is close to $M$. Note that by absolute continuity of $\mu$ we obtain $M(\mu,g)_{i,j}=M_{i,j}=0$ whenever at least one of $i$ and $j$ are in $B$. Let us therefore now assume $i,j\notin B$,  in which case we have
\[
(1-\varepsilon/ck) f_i\leq g_i \leq f_i+\varepsilon/c
\]
yielding
\begin{align*}
&M(\mu,g)_{i,j}-M(\mu,f)_{i,j}=
\int g_i\otimes g_j -f_i\otimes f_j \mr{d}\mu
\\
\geq{}&
\left(\left(1-\frac{\varepsilon}{ck}\right)^2-1\right)
M(\mu,f)_{i,j}\geq -\frac{2\varepsilon}{k}
\end{align*}
and
\begin{align*}
&M(\mu,g)_{i,j}-M(\mu,f)_{i,j}=
\int g_i\otimes g_j -f_i\otimes f_j \mr{d}\mu\\
\leq{}&
\frac{\varepsilon^2}{c}
+\frac{\varepsilon}{c}
\int
\left(
f_i\otimes 1
 +1\otimes f_j 
\right)
 \mr{d}\mu
\leq \frac{\varepsilon^2}{c}+\varepsilon.
\end{align*}
Thus if $\varepsilon<c$, we have $|M(\mu,g)_{i,j}-M_{i,j}|\leq 2\varepsilon$ whenever $i,j\notin B$, and so $\|M(\mu,g)-M\|_\infty\leq 2\varepsilon$, meaning that indeed $M\in\widetilde{C}(\mu,\beta)$.
\end{proof}

\begin{Le}
Given any measure $\mu\in\mc{M}^+_c$, positive integers $k<m$, a vector $\beta\in E_m$ with exactly $m-k$ vanishing coordinates and $\beta_0\in E_k$ denoting the vector obtained from this $\beta$ by erasing the $0$ coordinates, we have that $\widetilde{C}(\mu,\beta_0)$ can be obtained from
\[
\widehat{K}:=\left\{(\beta,M)\in\widehat{C}(\mu,m)\left|\beta_i\beta_j=0\Rightarrow M_{i,j}=0\right.\right\}
\]
by deleting the rows and columns pertaining to $0$ coordinates of $\beta$. The same can be achieved starting from the set
\[
\widetilde{K}:=\left\{(\beta,M)\in\widetilde{C}(\mu,[0,1]^m)\left|\beta_i\beta_j=0\Rightarrow M_{i,j}=0\right.\right\}.
\]
\end{Le}
\begin{proof}
Just as in the proof of Proposition \ref{Prop:zero}, it can be shown that any element of $\widehat{K}$ or $\widetilde{K}$ can be obtained as the limit of points coming from families $f\in\mk{F}_{\beta}$ with $\beta_i=0\Rightarrow f_i=0$.
\end{proof}

In other words, if we allow degenerating families of functions in our shapes, then lower dimensional shapes are retrievable from higher dimensional ones.

Different use of shapes to define convergence of measures may lead to different topologies, but the following result shows that in many cases, we obtain equivalent notions of convergence.

\begin{T}\label{Thm:conv_equiv}
Let $(\mu_n)_{\mb{N}^+}\subset \mc{M}^+_c$. The following are equivalent:
\begin{enumerate}
\item $C(\mu_n,k)$ converges for every $k\in\mb{N}^+$;
\item $\widetilde{C}(\mu_n,[\varepsilon(k),1]^k)$ converges for every $k\in\mb{N}^+$ for one/all $0<\varepsilon(k)\leq 1/k$;
\item $\widetilde{C}(\mu_n,\alpha)$ converges for every $m\in\mb{N}^+$ and $\alpha\in(0,1]^m\cap E_m$.
\end{enumerate}

\end{T}

\begin{proof}
First note that in $(2)$, given a positive integer $k$, convergence for an $\varepsilon_0(k)$ implies convergence for any larger $\varepsilon(k)$ by application of Lemma \ref{Le:unif_equi}.\\
The implications $(3)\Rightarrow(1)$ and $(2)\Rightarrow(1)$ are clear.\\
Let us now show $(2)\Rightarrow(3)$ with $\varepsilon(k)=1/2k$. Given an $m\in\mb{N}^+$ and $\alpha\in(0,1]^m\cap E_m$, let 
\[
k:=\lceil 1/\min\{\alpha_i\colon 1\leq i\leq m\} \rceil.
\]
Then we can find an $\alpha'\in[1/2k,1]^k$ that is a refinement of $\alpha$, i.e., there exists a surjection $P:\left\{1,\ldots,k\right\}\to\left\{1,\ldots,m\right\}$ with $\sum_{j\in P^{-1}(i)}\alpha'_j=\alpha_i$. To this surjection there corresponds a natural surjection 
$\widetilde{C}(\mu,\alpha')\to\widetilde{C}(\mu,\alpha)$ with Lipschitz bound $\max\{|P^{-1}(i)|^2\colon 1\leq i\leq m\}$.
Consequently, $\widetilde{C}(\mu_n,\alpha)$ is convergent whenever $\widetilde{C}(\mu_n,\alpha')$ is, but this holds true by Corollary \ref{Cor:cont}.\\
Finally, let us look at the implication $(1)\Rightarrow(2)$. Fix a positive integer $k$, a $\delta\in(0,1]$ and an $\varepsilon(k)\in(0,1/k]$. Let $\eta:=\delta\varepsilon(k)/4ck$, and suppose $K$ is a finite $\eta$-net in $[\varepsilon(k),1]^k\cap E_k$. Then by Lemma \ref{Le:unif_equi}, if for some $\nu_1,\nu_2\in\mc{M}^+_c$ we have
\[
\mr{dist}\left(\widetilde{C}(\nu_1,\alpha),\widetilde{C}(\nu_2,\alpha)\right)<\delta
\]
for each $\alpha\in K$, then
\[
\mr{dist}\left(\widetilde{C}(\nu_1,[\varepsilon(k),1]^k),\widetilde{C}(\nu_2,[\varepsilon(k),1]^k)\right)<\delta+4c(k-1)\eta/\varepsilon(k)<2\delta.
\]
Now set $m:=2k^2\lceil4c/\delta\varepsilon(k)\rceil$, and let 
\[
K:=\left\{
\alpha\in[\varepsilon(k),1]^k\cap E_k\left|
m\alpha_i\in\mb{N}^+\Forall 1\leq i\leq k
\right.
\right\}.
\]
Then $m>2k/\eta$, and $K$ is an $\eta$-net in $[\varepsilon(k),1]^k\cap E_k$. Thus to show convergence of $\widetilde{C}(\mu_n,[\varepsilon(k),1]^k)$, it suffices to show that for each $\alpha\in K$, the sequence $\widetilde{C}(\mu_n,\alpha)$ converges. But note that each such $\alpha$ is refined by $\beta:=(1/m,\ldots,1/m)\in\mb{R}^m$, and by the arguments used in $(2)\Rightarrow(3)$, this follows from the convergence of the sequence $C(\mu_n,m)=\widetilde{C}(\mu_n,\beta)$.
\end{proof}

\begin{D}
A set $(\mu_i)_{i\in I}$ of measures in $\mc{M}^+_c$ is said to be \emph{uniformly absolutely continuous with respect to the measure} $\nu\times\nu$ if for any $\varepsilon>0$ there exists a $\delta>0$ such that whenever $H\subset \mc{C}\times\mc{C}$ satisfies $\nu\times\nu(H)<\delta$, we have $\mu_i(H)<\varepsilon$ for any index $i\in I$.
\end{D}

\begin{R}
Note that the above definition is equivalent to each element of the set being absolutely continuous and the family of Radon-Nikodym derivatives being uniformly integrable with respect to $\nu\times\nu$. Also, by the de la Vall\'ee-Poussin theorem, if the Radon-Nikodym derivatives form a bounded set in $L^p$ for some $p>1$, then the sequence is automatically uniformly integrable.
\end{R}

For such families of measures, we have the following variant of Lemma \ref{Le:unif_equi}.

\begin{A}\label{Prop:unif_equi}
Let $U\subset \mc{M}^+_c$ be a family of uniformly absolutely continuous measures w.r.t. the measure $\nu\times\nu$. 
Then the maps
\[
\widetilde{C}(\mu,\cdot): E_m\to\mc{K}(\mb{R}^{m\times m})
\]
are uniformly equicontinuous on $U$.
\end{A}

\begin{proof}
By Lemma \ref{Le:unif_equi} it is enough to show equicontinuity in points $\beta\in E_m$ with at least one vanishing coordinate. Let again $B:=\{1\leq\ell\leq m|\beta_\ell=0\}$, and $\varepsilon\in(0,1)$, to which there by uniform absolute continuity corresponds a $\delta>0$. We may assume $\varepsilon<c$. Let $h:=\min\{\beta_i|i\notin B\}$, $\eta:=\min(\delta,\varepsilon h/cm)$, and $\alpha\in E_m$ such that $\|\alpha-\beta\|_\infty<\eta$. Choose a measure $\mu\in U$. Our aim is to show that $d_H\left(\widetilde{C}(\mu,\beta),\widetilde{C}(\mu,\alpha)\right)$ is small regardless of this choice.
Clearly $0<\alpha_i (1-\varepsilon/cm)<\beta_i$ for $i\notin B$. Thus, using the construction presented in the proof of the second part of Proposition \ref{Prop:zero}, given any family of functions $f\in\mk{F}_\alpha$, we can find a family $g\in\mk{F}_\beta$ such that for any $i,j\notin B$ we have $|M(\mu,f)_{i,j}-M(\mu,g)_{i,j}|<2\varepsilon$. Whenever at least one of $i,j$ is in $B$ -- say $j$ --, we have by absolute continuity that $M(\mu,g)_{i,j}=M(\mu,g)_{j,i}=0$, while
\begin{align*}
|M(\mu,f)_{j,i}|=|M(\mu,f)_{i,j}|=\int f_i\otimes f_j d\mu\leq \int 1\otimes f_j d\mu.
\end{align*}
However $\int 1\otimes f_j d\nu\times\nu=\alpha_j<\eta\leq\delta$, and also $0\leq 1\otimes f_j\leq 1$, thus by the uniform absolute continuity
\[
\int 1\otimes f_j d\mu<\varepsilon.
\]
For the other approximation, consider a family $g\in\mk{F}_\beta$, and let $A:=\{1\leq\ell\leq m|\beta_\ell>\alpha_\ell\}$. Clearly $A \cap B=\emptyset$. Define the following family of functions $f\in\mk{F}_\alpha$:
\begin{align*}
f_i:&=\frac{\alpha_i}{\beta_i}g_i{}&\mbox{ if } i\in A;\\
f_i:&=g_i+\frac{(\alpha_i-\beta_i)}{\sum_{j\notin A} \alpha_j-\beta_j}\sum_{\ell\in A}\frac{\beta_\ell-\alpha_\ell}{\beta_\ell}g_\ell{}& \mbox{ if } i\notin A.
\end{align*}
Note that by the choice of $\eta$, we have $\alpha_\ell/\beta_\ell\in (1-\varepsilon/cm,1)$ for all $\ell\in A$. Therefore for all $1\leq i\leq m$ we have
\[
(1-\varepsilon/cm) g_i\leq f_i\leq g_i + \varepsilon/cm,
\]
which again yields $\|M(\mu,f)-M(\mu,g)\|_\infty<2\varepsilon$ (cf. the calculations at the end of the proof of Proposition \ref{Prop:zero}). This means that $d_H(\left(\widetilde{C}(\mu,\beta),\widetilde{C}(\mu,\alpha)\right))<2\varepsilon$ whenever $\|\alpha-\beta\|_\infty<\eta$, with no dependence on $\mu\in U$.
\end{proof}

Combining this with Theorem \ref{Thm:conv_equiv}, we obtain the following result.
\begin{T}\label{Thm:conv_equiv_zero}
Let $U\subset \mc{M}^+_c$ be a uniformly absolutely continuous set of measures with respect to $\nu\times\nu$. Then for any sequence $(\mu_n)_{n\in\mb{N}^+}\subset U$, the following are equivalent:
\begin{enumerate}
\item $C(\mu_n,k)$ converges for every $k\in\mb{N}^+$;
\item $\widetilde{C}(\mu_n,[\varepsilon(k),1]^k)$ converges for every $k\in\mb{N}^+$ for one/all $0\leq\varepsilon(k)\leq 1/k$;
\item $\widetilde{C}(\mu_n,\alpha)$ converges for every $m\in\mb{N}^+$ and $\alpha\in[0,1]^m\cap E_m$.
\end{enumerate}

\end{T}

Finally, let us mention some easy observations regarding the topology induced by shape convergence.

\begin{Le}
Shape convergence in $\mc{M}^+_c$ is metrizable.
\end{Le}
\begin{proof}
For each $k\in\mb{N}^+$, consider the Hausdorff metric $D_k$ on the space $\mc{K}\left([0,c]^{k\times k}\right)$ of compact sets. Note that the $k$-shapes of $\mc{M}^+_c$ form a subset therein with $D_k$-diameter $2c$. Then the \emph{shape metric} 
\[
d_S(\mu_1,\mu_2):=\sum_{k=1}^\infty \frac{1}{2^k}D_k(C(\mu_1,k),C(\mu_2,k))
\]
induces shape convergence on $\mc{M}^+_c$.
\end{proof}

\begin{A}
We have that $d_S$ is a pseudo-metric on the set $\mc{M}^+_c$. If we factor out by the equivalence relation $\sim$ given by $0$ distance then $\mc{M}^+_c/\sim$ equipped with the shape metric $d_S$ is a compact Hausdorff space. 
\end{A}
\begin{proof}
The space $(\mc{M}^+_c/\sim,d_S)$ is a topological subspace of the compact Hausdorff space $\prod_{k=1}^\infty \left(\mc{K}\left([0,c]^{k\times k}\right), D_k\right)$. It is also closed, as any convergent sequence of matrices has a limit measure and any measure has a sequence of matrices converging to it (see Corollary \ref{cor:shpm2c}).
\end{proof}

\section{Some closed subsets of the shape spaces}\label{ch:closed}

In this section we take a closer look at some subsets of shape spaces, and show that they are closed sets.
Since the space $\mc{M}^+_c$ is compact for the shape convergence (and equivalent notions), these closed subsets are automatically compact themselves. Our aim is to rephrase the results presented in \cite{BCCZ,BCCZ2} on $L^p$ graphons and uniformly/$L^p$-upper regular sequences of graphons to the shape setting, in which compactness then is an easy corollary to the compactness of the whole space.
Since the shape metric is defined on a countable product space, a sufficient condition for a set to be closed is for it to have the form
\[
\left(\prod_{k=1}^\infty K_k\right)\bigcap \mc{M}^+_c,
\]
where each $K_k$ is a closed subset of $\mc{K}\left([0,c]^{k\times k}\right)$.

\begin{Le}\label{Le:shape_separable}
Let $\mu_1,\mu_2\in\mc{M}^+_c$ be two absolutely continuous measures with $d_S(\mu_1,\mu_2)=0$. Then also $\delta_\square(W_1,W_2)=0$, where $W_i$ is the Radon-Nikodym derivative of $\mu_i$ ($i=1,2$).
\end{Le}
\begin{proof}
By Theorem \ref{Thm:conv_equiv_zero} two absolutely continuous measures with the same shapes $(\widetilde{C}(\cdot,k))_{k\in\mb{N}^+}$ also have the same quotients $(\widetilde{C}(\cdot,[0,1]^k))_{k\in\mb{N}^+}$. But it was shown in \cite{BCCZ2} that two absolutely continuous measures with the same quotients have $\delta_\square$ distance zero hence the result follows.
\end{proof}

\begin{Le}\label{Le:p_ball_char}
Let $p>1$ be fixed. A measure $\mu\in\mc{M}^+_c$ is absolutely continuous with respect to $\nu\times\nu$ with Radon-Nikodym derivative $h\in L^p(\nu\times\nu)$ if and only if there exists $s\geq 0$ such that for each positive integer $k$ the shape $\widetilde{C}(\mu,k)\subset\mb{R}^{k\times k}$ lies within the closed $s/k^2$-ball of $\mb{R}^{k\times k}$.
\end{Le}

\begin{proof}
On the one hand, if $h$ has finite $L^p$ norm, then by contractivity and in light of the scaling used, each $\widetilde{C}(\mu,k)$ lies within the closed $\|h\|_p/k^2$-ball of $\mb{R}^{k\times k}$.\\
On the other hand, consider the consistent $2^n\times 2^n$ tables $T_n$ obtained from $\mu$ by restricting ourselves to the $\sigma$-algebras on $\mc{C}$ induced by the levels of the binary tree representation. Clearly $T_k\in \widetilde{C}(\mu,k)$. The functions $W_n$ defined by $n^2T_n$ are thus all in the closed $s$-ball of $L^p(\nu\times\nu)$. Let $h$ be a weak-* accumulation point of $\{W_n|n\in\mb{N}^+\}$. Then $\|h\|_p\leq s$, and it can easily be seen that the measure $h \cdot(\nu\times\nu)$ generates the same consistent tables $T_n$ as $\mu$ (the characteristic function of any Borel subset of $\mc{C}\times\mc{C}$ lies in $L^{p'}$). Since the $\sigma$-algebras of the tree levels generate the Borel $\sigma$-algebra on $\mc{C}\times\mc{C}$, these consistent tables extend uniquely, and so $\mu=h\cdot(\nu\times\nu)$.
\end{proof}

\begin{Cor}
For each $s\geq 0$ and $p>1$ the closed $s$-ball of $L^p(\nu\times\nu)$ within $\mc{M}^+_c$ is compact in the $d_S$ metric.
\end{Cor}

This result also implies that on the unit ball in $L^p$ ($p>1$), the shape and $\delta_\square$ topologies are equivalent. Indeed, for absolutely continuous measures, $\delta_\square$ convergence implies $d_S$ convergence, so the former induces a stronger topology. By Lemma \ref{Le:shape_separable}, the two topologies separate the same elements, so the quotient metric spaces live on the same sets. Then both metrics induce a compact Hausdorff topology, and since they are comparable, they are in fact identical.

The above does not hold for the case $p=1$, where an additional regularity is needed of families of $L^1$ functions to yield compactness results. This regularity is uniform integrability, or uniform absolute continuity if we speak of the corresponding measures. Therefore we next wish to show that this uniform absolute continuity can also be encoded in shapes.

Recall that a family $\Phi$ of finite measures on $(\Omega,\ms{B})$ is said to be \emph{uniformly absolutely continuous with respect to the finite measure} $\tau$ on $(\Omega,\ms{B})$ if for each $\varphi\in\Phi$ we have $\phi\ll\tau$, and the Radon-Nikodym derivatives $\left\{W_\varphi|\varphi\in\Phi\right\}$ satisfy that for any $\varepsilon>0$, there exists a $\delta>0$ such that
\[
B\in\ms{B}, \tau(B)<\delta\Rightarrow \int_B W_\varphi \mr{d}\tau<\varepsilon.
\]
In our case, when $(\Omega,\ms{B})$ is in fact $\mc{C}\times\mc{C}$ with its Borel $\sigma$-algebra, and $\tau=\nu\times\nu$, we may simply require
\[
B\in\ms{B}, \tau(B)<\delta\Rightarrow \varphi(B)<\varepsilon,
\]
from which absolute continuity and the existence of the Radon-Nikodym derivative follows.

This definition however will not directly be useful when looking at shapes, and we shall resort to an equivalent form.
By the de la Vall\'ee-Poussin theorem, a family $U\subset\mc{M}^+_c$ of measures is uniformly absolutely continuous with respect to $\nu\times\nu$ if and only if there exists $s\geq0$ and a non-negative increasing convex function $G:[0,\infty)\to[0,\infty)$ such that for any $\mu\in U$, its Radon-Nikodym derivative $W_\mu$ satisfies
\[
\int (G\circ W_\mu) \mr{d}\nu\times\nu\leq s.
\]
To simplify notation, let us write $G(\mu):=\int G\circ W_\mu \mr{d}\nu\times\nu$, and for any matrix $M\in\mb{R}^{m\times m}$, let $G(M):=\sum_{1\leq i,j\leq m} G(m^2M_{i,j})$.
Let further $U_{s,G}\subset \mc{M}^+_c$ denote the family of absolutely continuous measures for which the above inequality applies.

\begin{Le}\label{Le:unif_abs_cont_char}
For any $s\geq 0$ and non-negative increasing convex function $G:[0,\infty)\to[0,\infty)$, we have that $\mu\in U_{s,G}$ if and only if for all positive integers $k$ and $M\in\widetilde{C}(\mu,k)$, we have $G(M)\leq s$.
\end{Le}
\begin{proof}
By convexity of $G$, if $\mu\in U_{s,G}$, then for any positive integer $k$, the inequality
\[
G(M)\leq G(\mu)\leq s
\]
holds for each $M\in\widetilde{C}(\mu,k)$.\\
For the other direction, consider again the sequence of consistent tables $T_n\in \widetilde{C}(\mu,n)$ induced by the binary tree. We know that $G(T_n)\leq s$, and by the de la Vall\'ee-Poussin theorem, the measures they define on $\mc{C}\times\mc{C}$ are uniformly absolutely continuous. Their Radon-Nikodym derivatives $W_{T_n}$ then form a weakly compact set in $L^1(\nu\times\nu)$ by the Dunford-Pettis theorem, and thus have a weak accumulation point $h\in L^1(\nu\times\nu)$ which also satisfies 
\[
\int (G\circ h) \mr{d}\nu\times\nu\leq s
\]
Since each Borel subset of $\mc{C}\times\mc{C}$ has a characteristic function in $L^\infty$, it follows that $h\cdot(\nu\times\nu)$ induces the same consistent tables as $\mu$, and as in the previous Lemma, we may conclude that $\mu=h\cdot(\nu\times\nu)$.
\end{proof}

\begin{Cor}
For each $s\geq 0$ and non-negative increasing convex function $G:[0,\infty)\to[0,\infty)$, the family $U_{s,G}\subset\mc{M}^+_c$ is compact in the $d_S$ metric.
\end{Cor}
\begin{proof}
Since $G$ is convex and increasing, it is automatically continuous, hence for each positive integer $m$, the set
\[
H_m:=\left\{M\in\mb{R}^{m\times m}\left|G(M)\leq s\right.\right\}
\]
is closed. But then $K_m:=\left\{\widetilde{C}(\mu,m)\subset H_m\right\}\subset \mc{K}\left([0,c]^{m\times m}\right)$ is closed as well, and since
\[
U_{s,G}=\left(\prod_{k=1}^\infty K_k\right)\bigcap\mc{M}^+_c,
\]
we are done
\end{proof}

In other words, even though the closed unit ball of $L^1$ is not compact in this topology, families of uniformly absolutely continuous measures are relatively compact, and thus have accumulation points in the $d_S$ metric, each taking the form of a graphon.

The next steps take us beyond bounded balls in $L^p$ or uniformly absolute continuous families. The aim is to describe how singular measure sequences can be if they are to converge in shape to an $L^p$ or $L^1$ graphon.

We recall the following definitions from \cite{BCCZ}.

\begin{D}
A graphon $W:[0,1]^1\to\mathbb{R}$ is called $(C,\eta)$-\emph{upper} $L^p$ \emph{regular} if for any partition $\mathcal{P}$ of $[0,1]$ into measurable sets each of size at least $\eta$ we have that $\|W_\mathcal{P}\|_p\leq C$.\\
A sequence $(W_n)_{n\in\mb{N}^+}$ of graphons is called $C$-\emph{upper} $L^p$ \emph{regular} if for any $\eta>0$ there exists an index $N\in\mb{N}^+$ such that for any $n\geq N$ the graphon $W_n$ is $(C+\eta,\eta)$-upper $L^p$ regular.
\end{D}

We introduce a further regularity notion connected to shapes, better tailored to our purposes. Let $\mf{1}_k$ denote the $k\times k$ matrix with all entries equal to $1$.

\begin{D}
Given a positive integer $k$, an $\alpha\in(0,1]^k$ and a matrix $M\in\mb{R}^{k\times k}$, define
\[
\|M\|_{\alpha,p}:=\left(\sum_{1\leq i,j\leq k} \left(\frac{M_{i,j}}{\alpha_i\alpha_j}\right)^p\alpha_i\alpha_j\right)^{1/p}.
\]
Note that for a family of 0-1 valued functions $f\in\mk{F}_\alpha$ defining the partition $\mc{P}$ and a graphon $W$ we then have
\[
\|\mc{M}(\mu_W,f)\|_{\alpha,p}=\|W_\mc{P}\|_p.
\]
A measure $\mu\in\mc{M}^+_c$ is called \emph{shape} $(C,\eta)$-\emph{upper} $L^p$ \emph{regular} if for any positive integer $k$, $\alpha\in[\eta,1]^k\cap E_k$ and family $f\in\mk{F}_\alpha$ we have that $\|M(\mu,f)\|_{\alpha,p}\leq C
$.\\
A sequence $(\mu_n)_{n\in\mb{N}^+}\subset\mc{M}^+_c$ of measures is called \emph{shape} $C$-\emph{upper} $L^p$ \emph{regular} if for any $\eta>0$ there exists an index $N\in\mb{N}^+$ such that for any $n\geq N$ the measure $\mu_n$ is shape $(C+\eta,\eta)$-upper $L^p$ regular.\\
These notions naturally extend to graphons through their induced measures. 
\end{D}
Clearly shape regularity is stronger, as the functions in the family $f$ need not be 0-1 valued. However, it turns out that for sequences of absolutely continuous measures (graphons), these notions are actually equivalent.

We shall first prove the following strengthening of \cite[Proposition 2.10]{BCCZ}.

\begin{T}\label{Thm:shape_p_reg}
A sequence $(W_n)_{n\in\mb{N}^+}$ of graphons on $[0,1]^2$ that converges in $\delta_\square$-metric to an $L^p$ graphon $W$ is shape $\|W\|_p$-upper $L^p$ regular.
\end{T}
\begin{proof}
Fix $\eta\in(0,1)$, and an integer $k\leq 1/\eta$.
For any pair of functions $h_1,h_2:\mc{C}\to[0,1]$ and $n\in\mb{N}^+$, we have
\[
\left(
\int_{\mc{C}\times\mc{C}} (W-W_n)\cdot (h_1\otimes h_2) d\nu\times\nu
\right)
\leq
\|W_n-W\|_\square
\]
Thus for any family $f\in\mk{F}_\alpha$ with $\alpha\in[\eta,1]^k\cap E_k$ and $1\leq i,j\leq k$ we have that
\[
|\mc{M}(\mu_W,f)_{i,j}-\mc{M}(\mu_{W_n},f)_{i,j}|\leq \|W-W_n\|_\square,
\]
whence 
\begin{align*}
\|\mc{M}(\mu_{W_n},f)\|_{\alpha,p}&\leq\|\mc{M}(\mu_{W-W_n},f)\|_{\alpha,p}+\|\mc{M}(\mu_{W},f)\|_{\alpha,p}\\
&\leq
\|W-W_n\|_\square\cdot \|\mf{1}_k\|_{\alpha,p}+\|W\|_p.
\end{align*}
Thus, if we choose $N\in\mb{N}^+$ such that $\|W-W_n\|_\square\leq \eta/\|\mf{1}_k\|_{\alpha,p}$ for all $n\geq N$, the graphons $W_n$ ($n\geq N$) will be shape $(\|W\|_p+\eta,\eta)$-upper $L^p$-regular.
\end{proof}

This now allows us to prove the equivalence of the notions for sequences of graphons.

\begin{Cor}\label{Cor:shape_p_reg_equiv}
A sequence $(\mu_n)$ of absolutely continuous measures in $\mc{M}^+$ is $C$-upper $L^p$ regular if and only if it is shape $C$-upper $L^p$ regular.
\end{Cor}
\begin{proof}
Since the sequence is $C$-upper $L^p$ regular, it is eventually shape $(C+1,1)$-upper $L^p$ regular. If it isn't shape $C$-upper $L^p$ regular, we may find an $\varepsilon>0$ and a subsequence $\mu_{n_k}$ such that none of these measures is shape $(C+\varepsilon,\varepsilon)$-upper $L^p$ regular. However by \cite[Theorem 2.9]{BCCZ}, this subsequence has a further subsequence that converges to an $L^p$ graphon $W$ with $\|W\|_p\leq C$. By Theorem \ref{Thm:shape_p_reg} however, this second subsequence is then shape $\|W\|_p$-upper $L^p$ regular, leading to a contradiction.
\end{proof}

\begin{Le}\label{Le:shape_p_reg_char}
The set $R_{C,p,\eta}\subset\mc{M}^+_c$ of shape $(C,\eta)$-upper $L^p$ regular measures is compact in the $d_S$ metric. Its interior is given by 
\[
\bigcup_{0<\varepsilon<C} R_{C-\varepsilon,p,\eta}.
\]
\end{Le}
\begin{proof}
By definition, a measure $\mu\in\mc{M}^+_c$ is shape $(C,\eta)$-upper $L^p$ regular if and only if for any positive integer $k\leq1/\eta$, we have for all $(\alpha,M)\in\widetilde{C}\left(\mu,[\eta,1]^k\right)$ that $\|M\|_{\alpha,p}\leq C$. The space of compact subsets of the closed $C$-ball of $\ell^p(\mb{R}^{k\times k})$ is compact with respect to the corresponding Hausdorff metric, and so if we consider $\mc{M}^+_c$ as a topological subspace of the compact metric space
\[
\prod_{n=1}^\infty \left(\mc{K}\left([\varepsilon(n),1]^n\times[0,c]^{n\times n}\right)\right)
\]
with $\varepsilon(n):=\min\{1/n,\eta\}$, then the convergence on $\mc{M}^+_c$ is the one induced by the shapes $(\widetilde{C}(\cdot,[\varepsilon(n),1]^n))_{n\in\mb{N}^+}$. In this topology $R_{C,p,\eta}\subset\mc{M}^+_c$ is clearly closed.
However, by Theorem \ref{Thm:conv_equiv}, this topology on $\mc{M}^+_c$ is equivalent with the one induced by the metric $d_S$, implying the compactness. For the representation of the interior, note that it follows from the corresponding representations within in each fiber
\[
\mc{K}\left([\varepsilon(n),1]^n\times[0,c]^{n\times n}\right)
\]
with $n\leq 1/\eta$.
\end{proof}

\begin{Le}\label{Le:p_ball_regular}
Denoting the $C$-ball of $L^p(\nu\times\nu)$ within $\mc{M}^+_c$ by $V_{C,p}$, we have
\[
V_{C,p}=\bigcap_{\eta>0} R_{C+\eta,p,\eta}.
\]
\end{Le}
\begin{proof}
Clearly $V_{C,p}$ is contained in the intersection. Assume now that $\mu\in R_{C+\eta,p,\eta}$ for each $\eta>0$. Since
$R_{C+\eta',p,\eta'}\subset R_{C+\eta',p,\eta}\subset R_{C+\eta,p,\eta}$ for any $\eta>\eta'>0$, we have
\[
\mu\in \bigcap_{\eta'>0} R_{C+\eta',p,\eta}=R_{C,p,\eta} 
\]
for each $\eta>0$. By Lemma \ref{Le:p_ball_char}, we conclude that $\mu=h\cdot(\nu\times\nu)$ for some function $h$ with $\|h\|_p\leq C$. 
\end{proof}

This result, with the fact that the sets $R_{C+\eta,p,\eta}$ contain each other in their interiors, highlights the geometric reason why (shape) $L^p$-upper regularity is the correct property for shape convergence of measures (not necessarily absolutely continuous!) to an $L^p$ graphon.

Finally, let us turn our attention to uniform upper regularity, the $L^1$ analogue of the above. Recall the following definition from \cite{BCCZ2}.

\begin{D}[{\cite[Definition 5.1]{BCCZ2}}]
Given a function $K:(0,\infty)\to(0,\infty)$, a graphon $W$ is said to have $K$-\emph{bounded tails} if for each $\varepsilon>0$ we have
\[
\int \chi_{W\geq K(\varepsilon)}\cdot W
\mr{d}\nu\times\nu\leq \varepsilon.
\]
The graphon $W$ is said to be $(K,\eta)$-\emph{upper regular}, if for any partition $\mc{P}$ of $\mc{C}$ into sets of measure at least $\eta$, the step function $W_\mc{P}$ has $K$-bounded tails.
A sequence of graphons $W_n$ is called \emph{uniformly upper regular} if there exists a function $K$ and a sequence $\eta_n\to 0$ such that for each $n$, the graphon $W_n$ is $(K,\eta_n)$-upper regular.
\end{D}

Note that the required upper regularity means that the family of functions
\[
\ms{W}:=
\left\{
(W_n)_\mc{P}
|n\in\mb{N}^+,\, \mc{P} \mbox{ is a partition of }\mc{C} \mbox{ with all sets of measure at least } \eta_n
\right\}
\]
should have $K$-bounded tails. This however is just an equivalent formulation of uniform integrability for this family of functions, and using the de la Vall\'ee-Poussin theorem, we can define the following shape analogue.

\begin{D}
A sequence of measures $(\mu_n)_{n\in\mb{N}^+}\subset\mc{M}^+_c$ is said to be \emph{shape uniformly upper regular} if there exist a non-negative increasing convex function $G:[0,\infty)\to[0,\infty)$, a constant $s\geq0$ and a sequence $\eta_n\to0$ such that for any $n\in\mb{N}^+$, the familiy
\[
H_n:=
\left\{
M(\mu_n,f)\in\mb{R}^{k\times k}
\left|
k\in\mb{N}^+, \alpha\in[\eta_n,1]^k\cap E_k, f\in\mk{F}_\alpha
\right.
\right\}
\]
satisfies $G(M)\leq s+\eta$ for each $M\in H_n$. In this case the sequence is called $(s,G)$ \emph{shape upper regular}.\\
A measure $\mu\in\mc{M}^+_c$ is said to be $(s,G,\eta)$ \emph{shape upper regular} if the set
\[
H_{\mu,\eta}:=\bigcup_{k\in\mb{N}^+} 
\left\{
M(\mu,f)\in\mb{R}^{k\times k}
\left|
\alpha\in[\eta,1]^k\cap E_k, f\in\mk{F}_\alpha
\right.
\right\}
\]
satisfies $G(M)\leq s$ for all $M\in H_{\mu,\eta}$.
\end{D}
Let $U_{s,G,\eta}\subset\mc{M}^+_c$ denote the set of measures that are $(s,G,\eta)$ shape upper regular.
This means that a sequence of measures in $\mc{M}^+_c$ is shape uniformly upper regular if and only if there exist $G$ and $s$ such that the sequence eventually lies in $U_{s+\eta,G,\eta}$ for each $\eta>0$.

\begin{Le}\label{Le:shape_upper_reg_char}
For any $s\geq0, \eta>0$ and non-negative increasing convex function $G:[0,\infty)\to[0,\infty)$, the set $U_{s,G,\eta}\subset\mc{M}^+_c$ is compact in the $d_S$ metric. Its interior is given by
\[
\bigcup_{0<\varepsilon<s} U_{s-\varepsilon,G,\eta}.
\]
\end{Le}

\begin{proof}
The proof goes essentially as that of Lemma \ref{Le:shape_p_reg_char}, using the fact that for each positive integer $k$, the set
\[
\left\{M\in\mb{R}^{k\times k}\left|G(M)\leq s\right.\right\}
\]
is closed by continuity of $G$, and thus its closed subsets form a compact space for the Hausdorff metric, and that $(s,G,\eta)$ shape upper regularity can be characterized by the shapes $\widetilde{C}(\cdot,[\eta,1]^n)$ with $n\leq1/\eta$. For the interior, use that for $0<t<s$ the set
\[
\left\{M\in\mb{R}^{k\times k}\left|G(M)< t\right.\right\}
\]
is open, and so the closed sets that are contained in it form an open set in the Hausdorff metric.
\end{proof}

However, we also have the following analogue of Lemma \ref{Le:p_ball_regular}.

\begin{Le}
For any $s\geq0, \eta>0$ and non-negative increasing convex function $G:[0,\infty)\to[0,\infty)$, we have
\[
U_{s,G}=\bigcap_{\eta>0} U_{s+\eta,G,\eta}.
\]
\end{Le}
\begin{proof}
Clearly $U_{s,G}$ is contained in the intersection. If $\mu\in U_{s+\eta,G,\eta}$ for each $\eta$, then since
\[
U_{s+\eta',G,\eta'}\subset U_{s+\eta',G,\eta} \subset U_{s,G,\eta}
\]
for any $\eta>\eta'>0$, we have
\[
\mu\in\bigcap_{\eta'>0} U_{s+\eta',G,\eta}=U_{s,G,\eta}
\]
for each $\eta>0$. By Lemma \ref{Le:unif_abs_cont_char} we then indeed have $\mu\in U_{s,G}$.
\end{proof}

Since a single measure $\mu\in\mc{M}^+_c$ is uniformly absolutely continuous if and only if it is absolutely continuous, the above lemma has the following consequence.

\begin{Cor}
Any sequence $(\mu_n)_{\mb{N}^+}\subset\mc{M}^+_c$ that converges to an absolutely continuous measure $\mu\in U_{s,G}$ is $(s,G)$ shape upper regular. Also, any sequence $(\mu_n)_{\mb{N}^+}\subset\mc{M}^+_c$ that is $(s,G)$ shape upper regular for some $s$ and $G$ is relatively compact in the $d_S$ metric, and each accumulation point is an element of $U_{s,G}$.
\end{Cor}

\section{Further directions and remarks}\label{ch:conc}

\noindent{\bf The graph limit space $\mathcal{X}_s$:} Every graph limit notion has a corresponding graph limit space which consists of  equivalence classes of convergent sequences. Let $\mu_1$ and $\mu_2$ be two s-graphons. We say that $\mu_1$ and $\mu_2$ are isomorphic if $C(\mu_1,k)=C(\mu_2,k)$ holds for every $k\in\mathbb{N}$. In other words $\mu_1$ and $\mu_2$ are isomorphic if they represent the same limit object in the theory of s-convergence. The graph limit space $\mathcal{X}_s$ is the set of isomorphism classes of s-graphons. The space $\mathcal{X}_s$ is compact with respect to s-convergence. We can uniquely describe every element of $\mathcal{X}_s$ by the sequence $\{C(\mu,k)\}_{k=1}^\infty$ of the shapes of an arbitrary representative $\mu$ of the isomorphism class. 

\medskip

\noindent{\bf Partial order on $\mathcal{X}_s$:} We say that $\mu_1\preceq\mu_2$ if and only if $C(\mu_1,k)\subseteq C(\mu_2,k)$ holds for every $k$. It is clear that $\preceq$ gives a partial order on $\mathcal{X}_s$. A possible interpretation of $\preceq$ is that if $\mu_1\preceq\mu_2$ then $\mu_2$ represents a ``sparser limit object'' than $\mu_1$. For example the smallest element with respect to $\preceq$ is represented by the uniform measure on $[0,1]^2$ which is the limit of complete graphs. On the other hand we will see later that among regular s-graphons (regularity will be defined later) there is a maximal element which is the limit of the cycles $C_n$.

 \medskip

\noindent{\bf Uniqueness:}  It is natural to ask whether there is a more simple analytic characterization of isomorphism between s-graphons.
We don't have a satisfying answer to this question but it seems that the more singular $\mu_1$ and $\mu_2$ are the weaker statement we can make about their isomorphism. For example if $\mu_1$ is isomorphic to $\mu_2$ and they are absolutely continuous with bounded Radon-Nikodym derivatives then there are measure preserving maps $\psi_1,\psi_2:[0,1]\rightarrow [0,1]$ and a measure $\mu_3$ on $[0,1]^2$ such that $\mu_i$ is the push forward of $\mu_3$ with respect to $\psi_i^2$ for $i=1,2$. 

In general we don't have such a strong statement.  Let $\nu_\alpha$ denote the probability measure on $[0,1]^2$ obtained by first choosing a uniform element $x\in [0,1]$ and then choosing one of $(x+\alpha~ {\rm mod}~ 1,x)$ and $(x-\alpha~ {\rm mod}~ 1,x)$ with probability $1/2$. Note that $\nu_\alpha$ is a singular measure concentrated on a one dimensional set in $[0,1]^2$. Assume that $\alpha$ and $\beta$ are algebraically independent irrational numbers. Then one can check that $\nu_\alpha$ and $\nu_\beta$ are s-isomorphic but there is no $\psi_1$ and $\psi_2$ with the above property. 

\medskip

\noindent{\bf Convexity:} A closed subset in $\mathcal{X}_s$ is given by the isomorphism classes of s-graphons $\mu$ such that $C(\mu,k)$ is convex for every $k$. An interesting question is to understand what it means for a graph $G$ to be approximately convex in the sense that $G$ is close to $\mathcal{X}_s$ in some metrization of s-convergence.

\medskip

\noindent{\bf Dimension:} If an s-graphon is an absolutely continuous measure on $[0,1]^2$ then we can think of it as a $2$-dimensional object. This intuition is also related to the fact that if the number of edges in a convergent graph sequence has quadratic growth in the number of vertices then the sequence has always an absolutely continuous limit object. However sparser sequences may show a fractal like behavior. We attempt to associate a fractal dimension with s-convergent graph sequences (or equivalently with their limit objects)  through the limiting shapes. Let $\mu$ be an s-graphon. Each matrix $M\in C(\mu,k)$ is non negative and has entry sum $1$. This means that we can associate the entropy $\mathbb{H}(M):=\sum_{i,j}-\ln(M_{i,j})M_{i,j}$ with every $M\in C(\mu,k)$. Let $\mathbb{H}_k:=\min_{M\in C(\mu,k)}\mathbb{H}(M)$ and let $\dim(\mu):=\liminf_{k\to\infty}\mathbb{H}_k/\ln(k)$.

\medskip

\noindent{\bf Graphons vs. s-graphons}: A graphon (in the most restricted sense) is a symmetric measurable function of the form $W:[0,1]^2\rightarrow [0,1]$. The edge density $t(e,W)$ of $W$ is equal to the integral of $W$ according to the Lebesgue measure $\lambda^2$ on $[0,1]^2$. If $t(e,W)\neq 0$ then the function $W/t(e,W)$ is the Radon-Nikodym derivative of the s-graphon $\mu_W$ defined by $\mu_W(A):=t(e,W)^{-1}\int_A W d\lambda^2$ where $A$ is an arbitrary measurable set in $[0,1]^2$. Let $\{G_i\}_{i=1}^\infty$ be a graph sequence such that its limit in the sense of \cite{LSz} is the graphon $W$ and assume that $W\neq 0$. We have that $\{G_i\}_{i=1}^\infty$ is also s-convergent and its limit is $\mu_W$. In other words, if $\{G_i\}_{i=1}^\infty$ is a convergent dense graph sequence then the limit graphon can be recovered from its limiting s-graphon and the limiting edge density $\lim_{i\to\infty}t(e,G_i)$.

\medskip

\noindent{\bf Sampling and degree distribution:} One of the handicaps of s-convergence is that it does not seem to be naturally connected to any sampling procedure if the limit object has a singular part. A ``sampling type'' information that we can still recover in a greater generality is related to the degree distribution. Let $\mu$ be an s-graphon. If the marginal distribution $\mu'$ of $\mu$ on the first coordinate is uniform on $[0,1]$ then we say that $\mu$ is {\it regular}. It is clear that regularity is isomorphism invariant in our language as it can be completely recovered from the shapes. If $\mu'$ is absolutely continuous with respect to the Lebesgue measure then its value distribution is also an isomorphism invariant and we call it the degree distribution of $\mu$ in this case. Note that highly singular s-graphons can have an absolutely continuous marginal. For example the s-graphons $\nu_\alpha$ defined above are regular.

\noindent{\bf S-convergence of bounded degree graphs:}

It is clear from the definitions that for bounded degree graph sequences s-convergence is weaker than local-global convergence. The relationship to the Benjamini-Schramm limit is more complicated. None of them is weaker than the other. 
If a graph sequence $\{G_i\}_{i=1}^\infty$ is local-global convergent then the limit object (see also \cite{HLSz}) is a bounded degree measurable graph $G$ on the vertex set $[0,1]$ with the following measure preserving property. If $A,B\subseteq [0,1]$ are Borel measurable then $\int_A {\rm deg}_B(x) d\lambda=\int_B {\rm deg}_A(x) d\lambda$ where ${\rm deg}_U(x)$ denotes the number of neighbors of $x$ in $U$ and $\lambda$ is the usual Lebesgue measure on $[0,1]$.
Using this we have that the measure $\mu_G$ that is uniquely defined by $\mu_G(A\times B):=\int_A {\rm deg}_B(x) d\lambda$ is symmetric. It is not hard to see that the s-graphon $\mu_G$ is the limit object of $\{G_i\}_{i=1}^\infty$ in the sense of s-convergence. This consistence between the two limit objects is an encouraging fact. However the other direction does not work. The local-global limit cannot always be reproduced from the limiting s-graphon. It remains an interesting question to understand what s-convergence and the corresponding metric sees from bounded degree graphs.  

\medskip

\noindent{\bf Blow-up invariance:} 

It is useful to note that s-convergence is invariant with respect to the so-called blow-up operation. The $k$ blow-up of a graph $G$ is obtained from $G$ by replacing each vertex by $k$ vertices and each edge by a complete bipartite graph between the corresponding $k$-tuples of vertices. Subgraph densities and thus dense graph convergence is insensitive to this operation. It is easy to see that shapes of graphs are also invariant with respect to blow-up. 

\medskip

\noindent{\bf Relationship to other convergence notions for sparse graphs:} 

There are various limit notions for sparse graph sequences. We have already investigated the relationship between $L_p$ convergence (\cite{BCCZ},\cite{BCCZ2}) and s-convergence. Another related limit concept is the  {\it logarithmic convergence} described in \cite{Sz1}. Logarithmic convergence shares some common properties with s-convergence. For example logarithmic convergence is also blow-up invariant and it detects dense graph limit theory almost completely. Both s-convergence and logarithmic convergence loses a certain constant in the dense case but the meaning of these constants are different. In s-convergence we basically lose edge density: A graphon $W$ is equivalent with $cW$ where $c$ is any positive constant. In logarithmic convergence a graphon $W$ is equivalent with $\otimes^nW$ where $\otimes^nW$ is defined to be the unique graphon with the property that $t(H,\otimes^nW)=t(H,W)^n$ holds for every $H$. Note that logarithmic convergence was introduced to study problems in extremal combinatorics (such as Sidorenko's conjecture) which are invariant with respect to the $\otimes^n$ operation. A third limit concept for sparse graph sequences was introduced in \cite{Fr} by P. E. Frenkel. The main idea in \cite{Fr} is to use different normalizations of the homomorphsim numbers depending on the sparsity of the sequence. A nice fact abut this type of convergence is that it puts Benjamini-Schramm limits \cite{BS} and dense graph limits into a unified language. For other type of limit concepts see also \cite{NO}. At this point there is no ``best'' or ``strongest'' convergence notion for sparse graphs but there is a zoo of limit concepts capturing different properties. 

\noindent{\bf Stronger versions of s-convergence}

The notion of s-convergence can be strengthened and modified is various meaningful ways. One of the most obvious ones is the following. For two graphs $H,G$ we denote by ${\rm Hom}(H,G)$ the set of all graph homomorphisms from $H$ to $G$. Since ${\rm Hom}(H,G)$ is a set of functions from $V(H)$ to $V(G)$ we can consider ${\rm Hom}(H,G)$ as a subset of $V(G)^{V(H)}$. The characteristic function ${\rm Char}(H,G)$ of ${\rm Hom}(H,G)$ is a function of the form $V(G)^{V(H)}\rightarrow\{0,1\}$. In particular if $H$ is the single edge $e$ then ${\rm Char}(e,G)$ is the adjacency matrix of $G$. In general ${\rm Char}(H,G)$ is a $|V(H)|$ dimensional array. We say that $G_n$ is strongly s-convergent if the arrays ${\rm Char}(H,G_n)/|{\rm Hom}(H,G_n)|$ are convergent for every fixed $H$ is a similar sense as Definition \ref{def:matconv}. Following the philosophy of this paper, the limit object should be a collection of measures on the sets $[0,1]^{V(H)}$ where $H$ runs through all finite graphs. It is a non-trivial question to decide which systems of measures arise this way.

\section{Examples}\label{ch:ex}

\noindent{\bf Maximal regular s-graphon:} A symmetric matrix is regular if all the row sums (and thus the column sums) are equal. An s-graphon $\mu$ is regular if every $M\in C(\mu,k)$ is regular for every $k\in\mathbb{N}$.
We show that there is a maximal regular s-graphon $\mu_{\max}$ with respect to the partial order $\preceq$. Let $R_k$ denote the set of all non-negative symmetric $k\times k$ matrices with the property that all row sums are equal to $1/k$. It is clear that if $\mu$ is regular then $C(\mu,k)\subseteq R_k$. It is enough to show that there exists $\mu_{\max}$ with $C(\mu_{\max},k)=R_k$ for every $k$.
Let $C_n$ denote the cycle of length $n$. 

We claim that for every $k\in\mathbb{N}$ the sequence $\{C_0(C_n,k)\}_{n=1}^\infty$ converges to $R_k$ as $n$ goes to $\infty$ in the Hausdorff metric. To see this let $M\in R_k$. We have that $kM$ is a doubly stochastic matrix and thus it represents a symmetric Markov chain on $[k]=\{1,2,\dots,k\}$. Let $(a_1,a_2,\dots,a_n)\in [k]^n$ be a random vector obtained by running the Markov chain $kM$ for $n$ steps. Let $q_{i,j}$ denote the number of indices $t$ such that $(a_t,a_{t+1})=(i,j)$. We have that if $n$ is large then $q_{i,j}/n$ is close to $M_{i,j}$ with probability close to $1$ and the errors go to $0$ with $n$. This means that if we color the vertices of $C_n$ consecutively by $a_1,a_2,\dots,a_n$ then the number of edges between the $i$-th and the $j$-th color class is roughly $2nM_{i,j}$. It follows that the matrix of the partition normalized by $2|E(C_n)|=2n$ is close to $M$. With small changes we can make this partition balanced and thus our claim is proved.

Now let $\mu_{\max}$ be an s-graphon representing the limit of $\{C_n\}_{n=1}^\infty$. We obtained that $C(\mu_{\max},k)=R_k$. The reader can check that a concrete choice for $\mu_{\max}$ is the s-graphon $\nu_{\alpha}$ (where $\alpha$ is irrational) constructed in Chapter \ref{ch:conc}.

\medskip

\noindent{\bf Hypercubes and fat hypercubes:} Let $H_n$ denote the graph on the vertex set $\{0,1\}^n$ in which two vectors are connected if their Hamming distance is $1$. It can be proved that $\{H_n\}_{n=1}^\infty$ is convergent and the limit object is the maximal regular s-graphon $\mu_{\max}$ constructed above. The idea of the proof is the following. Let us chose a natural number $a$ such that $0<<a<< n$. We have that the map $f_n:v\to \langle v,1_n\rangle ~{\rm mod}~ a$ is a graph homomorphism from $H_n$ to $C_a$ where $1_n$ is the all $1$ vector of length $n$ and $v\in\{0,1\}^n$. We also have that $f_n$ is roughly balanced in the sense that preimages have similar size. From here we can use the preimages of the partitions of $C_a$ constructed above to show that $\lim_{n\to\infty}C_0(H_n,k)=R_k$.

For $0\leq\alpha\leq 1$ let $H_n^\alpha$ denote the graph on $\{0,1\}^n$ in which two vectors are connected if their Hamming distance is  in the interval $[\alpha n,\alpha n+3]$. We call them fat hypercubes. Fat hypercubes play an important illustrative role in the limit theory proposed in \cite{Sz1}. They provide natural examples for graph sequences in which the number of edges grows with a fixed power of the number of vertices which is between $1$ and $2$. There is a natural infinite version of $H_n^\alpha$ that we denote by $H^\alpha$. Let us choose a uniform random element $x$ in the cantor set $\mathcal{C}=\{0,1\}^\mathbb{N}$ and let $y$ be obtained from $x$ such that coordinate-wise independently with probability $\alpha$ we change the coordinate. The distribution $H^\alpha$ of the pair $(x,y)$ is symmetric on $\mathcal{C}\times\mathcal{C}$ and thus it is an s-graphon. We ask the following question: {\it Is it true that for fixed $\alpha$ the sequence $\{H_n^\alpha\}_{n=1}^\infty$ converges to $H^\alpha$?}

\medskip

\noindent{\bf Product graphs:} For two graphs $G_1$ and $G_2$ we denote by $G_1\times G_2$ the graph with vertex set $V(G_1)\times V(G_2)$ such that $(v_1,v_2)$ and $(w_1,w_2)$ are connected if and only if $(v_1,w_1)\in E(G_1)$ and $(v_2,w_2)\in E(G_2)$. For a fixed graph $G$ the sequence $\{G^n\}_{n=1}^\infty$ is a natural example for a sparse but not too sparse graph sequence. The limit object of this sequence is basically given by the uniform measure concentrated on the edges of $G^\infty$. Note that $V(G)^\infty$ is a Cantor set and the edge set $E(G^\infty)$ is a subset of $(V(G)^\infty)^2$. The uniform distribution on $E(G^\infty)$ is defined by choosing an infinite sequnce $(v_1,w_1),(v_2,w_2),\dots$ of independent, uniform random directed edges in $G$. 

\medskip

\noindent{\bf Subdivisions of complete graphs:} Let $K_n^\circ$ denote the $2$-subdivision of the complete graph $K_n$. The graphs $K_n^\circ$ are highly non-regular and their limit is non-trivial. We define a probability distribution $\mu^\circ$ on $\mathcal{C}^4=(\mathcal{C}^2)^2$ in the following way. We choose $(a,b)\in\mathcal{C}^2$ uniformly and then with probability $1/4$ each we choose one of $((a,b),(a,a)),((a,b),(b,b)),((a,a),(a,b)),((b,b),(a,b))$. It is clear that $\mu^\circ$ is symmetric with respect to exchanging the first two coordinates and the last two coordinates. Note that replacing $\mathcal{C}^2$ by $\mathcal{C}$ (using a continuous measure preserving bijection) we can also represent $\mu^\circ$ as a symmetric probability measure on $\mathcal{C}^2$. We claim that the limit of the graphs $K_n^\circ$ is $\mu^\circ$.

\subsection*{Acknowledgement.} The research leading to these results has received funding from the European Research Council under the European Union's Seventh Framework Programme (FP7/2007-2013) / ERC grant agreement n$^{\circ}$617747. The research was partially supported by the MTA R\'enyi Institute Lend\"ulet Limits of Structures Research Group.

\end{document}